\documentclass{article}

\usepackage{amsmath,amssymb,amsthm}
\usepackage{multirow}
\usepackage{url}

\newcommand{\vertiii}[1]{{\vert\kern-0.25ex\vert\kern-0.25ex\vert #1 
    \vert\kern-0.25ex\vert\kern-0.25ex\vert}}

\usepackage{natbib} 
\usepackage[hang,small,bf]{caption}
\usepackage{graphicx}

\usepackage[CJKbookmarks=true,
            bookmarksnumbered=true,  
			bookmarksopen=true,
			colorlinks=true,
			citecolor=blue,
			linkcolor=blue,
			anchorcolor=red,
			urlcolor=blue]{hyperref}
\usepackage[capitalize,noabbrev]{cleveref}
\usepackage{todonotes}

\usepackage{accents}
\newlength{\dhatheight}

\usepackage[top=2cm, bottom=2cm, left=2.8cm,
right=2.8cm]{geometry} 

\newcounter{rcnt}[section]

\def\argmin{\mathop{\rm argmin}}

\newcommand{\seqn}{\textnormal{seq}_{[n]}}
\newcommand{\indep}{\perp \!\!\! \perp}
\newcommand{\radius}{\mathsf{rad}}

\newcommand{\dtv}{\mathsf{d}_{\mathsf{TV}}}

\newcommand{\Mod}[1]{\ (\mathrm{mod}\ #1)}

\newcommand{\R}{\mathbb{R}}

\newcommand{\EE}{\mathbb{E}}
\newcommand{\E}{\mathbb{E}}

\newcommand{\PP}{\mathbb{P}}

\newcommand{\III}[1]{{\left\vert\kern-0.25ex\left\vert\kern-0.25ex\left\vert #1 
    \right\vert\kern-0.25ex\right\vert\kern-0.25ex\right\vert}}

\newcommand{\simiid}{\stackrel{\text{iid}}{\sim}}

\newcommand{\ca}{\mathcal{A}} \newcommand{\cb}{\mathcal{B}}  \newcommand{\cd}{\mathcal{D}}    \newcommand{\ch}{\mathcal{H}}         \newcommand{\cq}{\mathcal{Q}}     \newcommand{\cw}{\mathcal{W}} \newcommand{\cx}{\mathcal{X}}  \newcommand{\cz}{\mathcal{Z}}

\newcommand{\eps}{\varepsilon}

\makeatletter
\providecommand*{\diff}%
	{\@ifnextchar^{\DIfF}{\DIfF^{}}}
\def\DIfF^#1{%
	\mathop{\mathrm{\mathstrut d}}%
		\nolimits^{#1}\gobblespace}
\def\gobblespace{%
	\futurelet\diffarg\opspace}
\def\opspace{%
	\let\DiffSpace\!%
	\ifx\diffarg(%
		\let\DiffSpace\relax
	\else
		\ifx\diffarg[%
			\let\DiffSpace\relax
	\else
		\ifx\diffarg\{%
			\let\DiffSpace\relax
		\fi\fi\fi\DiffSpace}

\theoremstyle{plain}
\newtheorem{theorem}{Theorem}
\newtheorem{prop}[theorem]{Proposition}

\newtheorem{corollary}[theorem]{Corollary}
\newtheorem{lemma}[theorem]{Lemma}

\newtheorem{assumption}{Assumption}

\theoremstyle{definition}
\newtheorem{definition}[theorem]{Definition}

\theoremstyle{remark}
\newtheorem{remark}[theorem]{Remark}

\usepackage[blocks, affil-it]{authblk}
\title{Stability via resampling: statistical problems beyond the real line}
\author[1]{Jake A. Soloff}
\author[1]{Rina Foygel Barber}
\author[1,2]{Rebecca Willett}
\affil[1]{Department of Statistics, University of Chicago}
\affil[2]{Department of Computer Science, University of Chicago}
\date{\today}

\begin{document}

\maketitle

\begin{abstract}
        Model averaging techniques based on resampling methods (such as bootstrapping or subsampling) have been utilized across many areas of statistics, often with the explicit goal of promoting stability in the resulting output. 
    We provide a general, finite-sample theoretical result guaranteeing the stability of 
    bagging when applied to algorithms that return outputs in a general space, so that the output is not necessarily a real-valued---for example, an algorithm that estimates a vector of weights or a density function.
    We empirically assess the stability of bagging on synthetic and real-world data for a range of problem settings, including causal inference, nonparametric regression, and Bayesian model selection. 
\end{abstract}

\section{Introduction}\label{sec:intro}

Statistical stability holds when our conclusions, estimates, fitted models, predictions, or decisions are insensitive to small perturbations of data. Stability is a pillar of modern data science \citep{yu2020veridical}. Nonparametric tools for statistical inference---e.g. the bootstrap, subsampling, the Jackknife---perturb data to assess variability and can thus be viewed as aspects of stability analysis \citep{yu2013stability}. 
While assessing stability is in general statistically intractable without placing assumptions on the data~\citep{kim2021black}, 
these resampling-based approaches can also be used to reduce the issue of instability: 
for instance, \cite{breiman1996bagging,breiman1996heuristics} originally proposed bagging---which applies bootstrap resampling to construct an ensemble of machine learning models---as an off-the-shelf stabilizer for regression and classification. 
In this paper, we show how bagging can be used to stabilize statistical methods far beyond its origins in supervised learning problems.

\subsection{Problem settings}\label{sec:problem_settings}
We begin by highlighting some specific statistical problems where provably stable methods are desirable.

\paragraph{Problem setting 1 (background): the supervised learning setting}

In the supervised learning setting, we are given a labeled data set $(X_1,Y_1),\dots,(X_n,Y_n)$ and are tasked with constructing a regression function $\hat{f}$ to predict $Y$ given $X$ for new data points. Many modern methods from the recent machine learning literature offer powerful predictive performance, but suffer from high sensitivity to the training data, where a prediction $\hat{f}(X)$ (at a new test point $X$) may vary immensely if the training data is perturbed slightly. Interestingly, both high accuracy and poor stability can often be a consequence of working with highly overparameterized models. 

Suppose that the response $Y$ is real-valued, so that our regression functions take values in $\mathbb{R}$, i.e., $\hat{f}: \mathcal{X}\to\mathbb{R}$. Recent work by \cite{soloff2024bagging} establishes that bagging the predictive model---that is, re-fitting $\hat{f}$ on bootstrapped training sets, and then returning the averaged prediction---yields an assumption-free guarantee on the stability of the prediction $\hat{f}(X)\in\mathbb{R}$.

While \cite{soloff2024bagging}'s work is presented as a result for supervised learning, it can be viewed more generally as a guarantee of stability for \emph{any} algorithm that is trained on data and returns a real-valued output. But even with this generalization,
this result only covers a very specific statistical scenario. For many statistical problems, we might be interested in the stability of a procedure that returns a more complex object, such as a function or a vector of weights, rather than a real-valued output.

\paragraph{Problem setting 2: reproducible Bayesian inference}
In Bayesian data analysis, given a data set $\cd$ containing $n$ observations drawn from some distribution $p_\theta$,  we are interested in the posterior distribution of the parameter $\theta$. For instance, $\theta$ might represent a particular model from a finite set of possibilities, in which case the posterior distribution on $\theta$ is simply a vector of weights on this finite set. Alternatively, $\theta$ might take values in a continuous space such as $[0,1]$ (e.g., if $\theta$ represents a probability) or $\R^d$ (e.g., coefficients for some regression model), in which case the posterior distribution is a  distribution over $[0,1]$ or over $\R^d$---a more complex object.

Bayesian inference may sometimes lead to answers that are highly sensitive with respect to perturbations of the data. In particular, this can occur in settings where the model is misspecified (but nonetheless, is a reasonably good approximation to the observed data, so that computing the posterior is still a reasonable goal). To alleviate this instability, the \textit{BayesBag} approach \citep{buhlmann2014discussion} averages posterior distributions over different resamplings of the data~$\cd$. \cite{huggins2023reproducible} analyzes the accuracy and stability properties of the BayesBag approach for the specific problem of Bayesian model selection, where $\theta$ lies in a finite set of models $\mathfrak{M} = \{\mathfrak{m}_1,\dots,\mathfrak{m}_K\}$ 
(and so the posterior on $\theta$ is given by a vector of weights of length $K$, i.e., with the $k$th weight indicating the posterior probability $\mathbb{P}\{\theta = \mathfrak{m}_k\mid \cd\}$). 
Their work establishes BayesBag as an empirically effective method with asymptotic guarantees, but has not yielded finite-sample guarantees even in the finite model setting.
Of course, we can also ask this question even more generally: will BayesBag lead to greater stability in a general setting, where $\theta$ may lie in an arbitrary space (e.g., if the posterior is a distribution on $[0,1]$, or on $\R^d$, etc.)?

\paragraph{Problem setting 3: synthetic controls for causal inference}
The synthetic control method (SCM) is a popular approach for estimating causal effects with panel data \citep{abadie2003economic,abadie2010synthetic,abadie2015comparative,abadie2021using}. The method is particularly useful in settings with a single treated unit and a relatively small number of control units. A synthetic control is a weighted combination of control units, where the weights are chosen to match the treated unit's pre-treatment outcomes and other covariates. SCM estimates the counterfactual---what outcomes would we expect for the treated unit had the policy not been adopted?---using the synthetic control's post-treatment outcomes. For example, \cite{abadie2015comparative} applies SCM to the 1990
German reunification and its impact on per capita GDP in West Germany. Per capita GDP 
was increasing prior to reunification and continued to increase in subsequent years, for this was a period of sustained growth across industrialized countries. To assess the counterfactual, a `synthetic West Germany' is constructed as follows: we find the convex combination of other industrialized countries such that the pre-1990 economic trend most closely resembles that of West Germany. \cite{abadie2015comparative} finds that West Germany's economy did not grow as much as it would have without reunification.

Often SCM's weights are of direct interest, as they say something about which peer countries are most relevant in this counterfactual assessment. 
However, in order to reliably interpret the weights, we should first ensure they are stable \citep{murdoch2019definitions}. For instance, if we remove \emph{Denmark} from the analysis 
(a country which receives zero weight in the synthetic West Germany) then there is no longer a positive weight on the United States, which previously received $22\%$ of the weight. This sensitivity appears to betray our intuitive interpretation of the meaning of SCM's weights. 
Can a resampling-type approach to SCM alleviate these issues?

\subsection{Our contributions}
Algorithmic stability is desirable for statistical methods far beyond its origins in statistical learning theory \citep{devroye1979distribution,devroye1979distribution2,bousquet2002stability}. Broadly, stability is considered a `prerequisite for trustworthy interpretations' \citep{murdoch2019definitions}. The goal of this paper is to establish quantitative guarantees on the stability of bagging as a general-purpose tool across a wide range of problems where the output---the output being estimated from the data---might not be real-valued and might even be infinite-dimensional. 
\Cref{tab:applications} lists a range of problem settings where our framework can be directly applied (including the settings discussed above, namely, supervised learning, Bayesian inference, and synthetic controls). Formally, our results will apply to any algorithm, randomized or deterministic, with outputs in bounded subset of a separable Hilbert or Banach space---we give details below. We highlight one particularly important finding: our main results will show that for a Hilbert space, there is \emph{no cost} to high dimensionality in terms of the stability guarantees that can be obtained (in contrast, for the more general setting of a Banach space, this is no longer the case).

\paragraph{Outline} The remainder of this paper is organized as follows.
In \Cref{sec:framework}, we define a framework for algorithmic stability in any metric space, and then define bagging for an arbitrary algorithm in \Cref{sec:bagging}. Our main results on the stability of bagging are presented in \Cref{sec:hilbert_theory} and \Cref{sec:banach_theory}.   In \Cref{sec:apps}, we take a closer look at the implications of our results on a range of specific settings appearing in \Cref{tab:applications}, and discuss some consequences of our stability guarantee. We conclude with a discussion in \Cref{sec:conclusion}.
All proofs are deferred to the Appendix.

\begin{table}[t]
    \centering
    \begin{tabular}{c|c|c}
       Applications & Output Space $\cw$ & Metric $\mathsf{d}$ on $\cw$\\
        \hline\hline
        Supervised learning & Predicted value range $[a,b]\subset\R$ & Absolute value $|\cdot|$\\
        \hline
    \begin{tabular}{@{}c@{}}
        Bayesian model selection \\
        Synthetic controls 
        \end{tabular} & Probability simplex $\Delta_{K-1}$ & \begin{tabular}{@{}c@{}} Euclidean dist.\ $\|\cdot\|_2$ or 
        \\ Total variation dist.\ $\dtv$
        \end{tabular} \\
    \hline
    \begin{tabular}{@{}c@{}}Constrained Lasso ($q=1$)\\ Constrained Ridge ($q=2$)
        \end{tabular} & $\ell_q$-ball $\mathbb{B}_q^{d}(R)$ & $\ell_q$-norm $\|\cdot\|_q$ \\
    \hline
    \begin{tabular}{@{}c@{}}Density estimation\\ (e.g., Bayesian posterior)
        \end{tabular} & $L$-Lipschitz densities on $[0,1]$ & Total variation dist.\ $\dtv$ \\
    \hline
    Nonparametric regression & Sobolev ball of functions $\R^d\to\R$ & Sobolev norm $\|\cdot\|_{s,2}$ \\
    \hline
    \end{tabular}
    \caption{Example problem settings for our stability framework, including the examples highlighted in Problem Settings~1--3 (in \Cref{sec:intro}).}
    \label{tab:applications}
\end{table}

\section{Framework: algorithmic stability in a metric space}\label{sec:framework}

We define a \emph{randomized algorithm}~$\ca$ as function that takes in a finite sequence of data $\cd = (Z_i)_{i\in [n]}$, as well as some independent randomness $\xi\sim \text{Uniform}([0,1])$, and returns an output~$\ca(\cd; \xi)$. 
The term $\xi$, which we refer to as the \emph{random seed}, determines all randomness used by the algorithm---for instance, a random initialization or a random split of the data into batches.
The algorithm $\ca$ is deterministic if it does not depend on the seed~$\xi$.

This notion of a (deterministic or randomized) algorithm covers many disparate problems in statistical learning, such as the problem settings introduced in \Cref{sec:problem_settings}. 
\begin{itemize}
    \item Consider a supervised learning setting (Problem Setting 1) with data points $Z_i = (X_i,Y_i)\in\cx\times\R$. If we are interested in predicting $Y$ given a new feature $X=x$, our algorithm $\ca$ can return $\hat{f}(x)\in\R$, the predicted value at $X=x$, as its output. (Alternatively, if we are interested in building a  model for $Y$ given $X$, our algorithm $\ca$ can return an entire fitted regression function $\hat{f}$---a more complex output.)
    \item In Bayesian inference (Problem Setting 2), we observe data points $Z_1,\dots,Z_n$ drawn from some distribution $p_\theta$, where $\theta\in\Theta$ is the unknown parameter. To perform Bayesian inference on $\theta$, our algorithm $\ca$ would then return the posterior distribution over $\theta\in\Theta$. 
    \item For the problem of synthetic controls (Problem Setting 3), our data points $Z_i$ are the feature vectors for each control unit $i=1,\dots,n$, and the algorithm $\ca$ returns a weight vector of dimension $n$, specifying the convex combination of control units that define the synthetic control.
\end{itemize}
In some settings it may be common to use a randomized, rather than deterministic, algorithm. For example, in Bayesian inference,
in a high-dimensional setting we might choose to approximate the posterior through (random) sampling rather than computing the density explicitly.

Formally, we write
\[
\ca : \bigcup_{n\ge 0}\cz^n\times [0,1] \to \cw,
\]
where the output takes values in some space $\cw$ equipped with a metric $\mathsf{d}$---see \Cref{tab:applications} for examples. From this point on, we write $\hat{w} = \ca(\cd;\xi)$ to denote the output, regardless of the nature of the output (e.g., a finite-dimensional vector of weights for the synthetic control problem, or a posterior density function for the Bayesian inference setting, etc.).\footnote{We assume throughout that $t\mapsto \ca((Z_i)_{i\in[n]},t)$, which is a function from $[0,1]$ to $\cw$, is measurable for any fixed $(Z_i)_{i\in[n]}\in\cz^n$. Measurability is defined with respect to the Borel $\sigma$-algebra on $(\cw, \mathsf{d})$.}

\subsection{Defining stability} There are many related definitions of algorithmic stability in the literature. We focus on stability of the algorithm's output with respect to dropping a data point: will the output $\hat{w}$ of the algorithm be similar if one data point is removed at random from the training data set?
\begin{definition}
\label{def:bstability}
    An algorithm~$\ca$ has \emph{mean-square stability} $\beta^2$ with respect to a metric~$\mathsf{d}$ if, for all data sets $\cd \in \cz^n$ of size $n$,
    \[
    \frac{1}{n}\sum_{i=1}^n \E_\xi\left[\mathsf{d}^2\left(\hat{w}, \hat{w}^{\setminus i}\right)\right]\le \beta^2.
    \]
    where $\hat{w}=\ca(\cd; \xi)$ is the output of the algorithm run on the full data set~$\cd$, $\hat{w}^{\setminus i}=\ca(\cd^{\setminus i}; \xi)$ is the output of the algorithm run on the leave-one-out data set $\cd^{\setminus i} = (Z_j)_{j\ne i}$, and the expectation averages over $\xi\sim \text{Uniform}[0,1]$.
\end{definition}

This first definition captures the average perturbation in $\hat{w}$ when one data point is removed at random from some data set $\cd\in\cz^n$---that is, $\mathsf{d}^2\left(\hat{w}, \hat{w}^{\setminus i}\right)$, averaged over which data point $i$ gets removed.
An alternative way to define stability is to control the tails of these perturbations $\mathsf{d}\left(\hat{w}, \hat{w}^{\setminus i}\right)$, ensuring that not too many of them can be large.
\begin{definition}\label{def:stability}
An algorithm~$\ca$ has \emph{tail stability} $(\varepsilon, \delta)$ with respect to a metric~$\mathsf{d}$ if, for all data sets $\cd \in \cz^n$ of size $n$,
    \[
    \frac{1}{n}\sum_{i=1}^n \mathbb{P}_\xi\left\{\mathsf{d}(\hat{w}, \hat{w}^{\setminus i}) \ge \eps\right\}\le\delta,
    \]
where $\hat{w}=\ca(\cd; \xi)$ and $\hat{w}^{\setminus i}=\ca(\cd^{\setminus i}; \xi)$ as in \Cref{def:bstability}.
\end{definition}
The terminology ‘$\ca$ has mean-square stability $\beta^2$’ in \Cref{def:bstability},  and ‘$\ca$ has tail stability $(\varepsilon, \delta)$’ 
in \Cref{def:stability},
suppresses the dependence on the sample size~$n$. Since we perform a non-asymptotic stability
analysis, we treat~$n$ as a fixed positive integer throughout.

These two notions of stability are clearly very similar in flavor, and the following simple result shows that each implies the other.
\begin{prop}\label{prop:stability_defs}
     If $\ca$ has  mean-square stability $\beta^2$ with respect to $\mathsf{d}$, then $\ca$ has tail stability $(\varepsilon,\delta)$ with respect to $\mathsf{d}$, for any $\varepsilon,\delta\geq 0$ satisfying $\delta\varepsilon^2\geq \beta^2$. Conversely, if $\ca$ has tail stability $(\varepsilon,\delta)$ with respect to~$\mathsf{d}$, then $\ca$ has mean-square stability $\beta^2 = \eps^2 + \delta\,\sup_{w,w'\in\cw}\mathsf{d}^2\left(w, w'\right)$ with respect to~$\mathsf{d}$.
\end{prop}
\begin{proof}
    The first claim follows by Markov's inequality, since
    \begin{align*}
    \frac{1}{n}\sum_{i=1}^n \mathbb{P}_\xi\left\{\mathsf{d}\left(\hat{w}, \hat{w}^{\setminus i}\right) \ge \eps\right\}
    &= \mathbb{E}_\xi\left[\frac{1}{n}\sum_{i=1}^n \mathbf{1}\left\{\mathsf{d}^2\left(\hat{w}, \hat{w}^{\setminus i}\right) \ge \varepsilon^2\right\}\right] \\
    &\leq \varepsilon^{-2} \mathbb{E}_\xi\left[\frac{1}{n}\sum_{i=1}^n \mathsf{d}^2\left(\hat{w}, \hat{w}^{\setminus i}\right)\right] \leq \varepsilon^{-2}\beta^2.
    \end{align*}
    For the second claim, since $\mathsf{d}^2\left(\hat{w}, \hat{w}^{\setminus i}\right) \le \varepsilon^2 + \mathbf{1}\left\{\mathsf{d}\left(\hat{w}, \hat{w}^{\setminus i}\right) \ge \varepsilon\right\}\cdot \sup_{w,w'\in\cw}\mathsf{d}^2\left(w, w'\right)$, we have
    \begin{align*}
    \frac{1}{n}\sum_{i=1}^n \E_\xi\left[\mathsf{d}^2\left(\hat{w}, \hat{w}^{\setminus i}\right)\right]
    &\le \eps^2 + \frac{1}{n}\sum_{i=1}^n \PP_\xi\left\{\mathsf{d}\left(\hat{w}, \hat{w}^{\setminus i}\right) \ge \varepsilon\right\}\cdot \sup_{w,w'\in\cw}\mathsf{d}^2\left(w, w'\right)\\
    &\le \eps^2 + \delta\cdot \sup_{w,w'\in\cw}\mathsf{d}^2\left(w, w'\right),
    \end{align*}
    completing the proof.
\end{proof}

\subsection{Related notions of stability}
Our definitions of stability are stronger than many common definitions in the literature. For instance, we require the algorithm to be stable for \emph{any} fixed data set~$\cd$ to certify the stability of the method without looking at the specific data set. \Cref{def:stability} implies in particular that the algorithm is stable when the data~$(Z_i)_{i\in [n]}$ are i.i.d.\ from some distribution~$P$, but our definition does not require any stochastic model.

In many cases, practitioners care specifically about the stability of the algorithm's output $\hat{w}$, e.g., if we want to interpret the values of some learned weights, cluster assignments, etc. However, in other cases, we may be more concerned with the stability of some downstream decisions or outcomes based on an analysis. For instance, in supervised learning (Problem Setting 1), we may desire stability of a loss incurred by the output $\hat{w}$---e.g., a predictive error $\ell(\hat{w},(x,y)) = |y-\hat{w}(x)|$. 
These types of loss-based stability properties can often be inferred from stability of the output $\hat{w}$ (see \cite{soloff2024bagging} for further discussion). In this paper, we only consider the stability of $\hat{w}$ itself.

\section{Constructing a bagged algorithm via resampling}\label{sec:bagging}
In this section, we formally define the construction of a bagged algorithm.

Stability is about (in)sensitivity to perturbed data sets. Resampling methods deliberately perturb the data, so they are naturally a useful tool for stabilizing algorithms: given a data set $\cd\in\cz^n$ containing $n$ data points, we might retrain our algorithm repeatedly on resampled data sets,  
and then average the resulting outputs. 
There are several common  schemes for constructing these resampled data sets. 
Here, we mention two of the most common methods: 
\begin{itemize}
    \item \emph{Bootstrapping}  
    constructs each resampled data set by sampling $m$ indices $i_1,\dots,i_m$ uniformly with replacement from $[n]$ (the term ``bootstrap'' commonly refers to this method with the choice $m=n$, so that the resampled data sets are of the same size as the original~$\cd$). 
    \item Alternatively, we may instead sample data points without replacement, so that each contains a uniformly drawn subset containing $m$ of the original $n$ data points (e.g., with $m=n/2$). This type of approach is known as \emph{subbagging}.
\end{itemize}

We now define some notation that unifies these two variants (as well as allowing for other resampling schemes, as we discuss below).

Following the framework and notation of \cite{soloff2024bagging}, we define the set of finite sequences of indices in~$[n]$:
\[\seqn := \{(i_1, \ldots, i_k) : k\ge 0, \,i_1,\ldots,i_k\in [n]\}.\] 
A sequence $r=(i_1,\dots,i_m)\in\seqn$ is commonly referred to as a \emph{bag} in the context of data set resampling. For any $r\in\seqn$, 
define $\cd^r = (Z_i)_{i\in r}\in \cz^m$, the corresponding data set. 
Note that if the bag $r$ contains repeated indices (i.e., $i_k = i_\ell$ for some $k\neq \ell$), then the same data point from the original data set $\cd$ will appear multiple times in $\cd^r$.

We are now ready to define  the bagged version of any base algorithm $\ca$, using a particular resampling distribution $\cq_n$.\footnote{While the term ``bagging'' is sometimes used as a synonym for bootstrap averaging (i.e., averaging outputs after sampling from the data set without replacement) in the literature, in this paper we will use ``bagging'' more broadly, allowing for any resampling distribution~$\cq_n$.} From this point on, we assume that $\cw$ is a closed and compact subset of a separable Banach space---this technical condition ensures that taking an average or an expected value of outputs in $\cw$ is well-defined.
\begin{definition}[Bagging for an algorithm $\ca$]\label{def:bag}
   Given a data set~$\cd \in \cz^n$ and a distribution~$\cq_n$ on $\seqn$,
   return the output
   \[\hat{w} = \frac{1}{B}\sum_{b=1}^B \ca(\cd^{r_b};\xi_b),\]
where $(r_b,\xi_b)\stackrel{\textnormal{iid}}{\sim}\cq_n\times\textnormal{Uniform}[0,1]$.
\end{definition}
\noindent For example, bootstrapping corresponds to choosing $\cq_n$ as the uniform distribution over all sequences of length $m$, while subbagging takes $\cq_n$ as the uniform distribution over sequences of $m$ \emph{distinct} elements.

For our theoretical analysis, it will be useful to consider taking $B\to \infty$, so as to ``derandomize'' the procedure.
\begin{definition}[Derandomized bagging for an algorithm $\ca$]\label{def:derand-bag}
   Given a data set~$\cd \in \cz^n$ and a distribution~$\cq_n$ on $\seqn$,
   return the output
   \[\hat{w} = \EE_{r,\xi}[\ca(\cd^r; \xi)],\]
where the expected value is computed with respect to $(r,\xi)\sim\cq_n\times \text{Uniform}[0,1]$.  
\end{definition}

\subsection{A general framework for bagging}
The two variants discussed above, bagging and subbagging, are special cases of a broader framework, allowing for more flexibility in the choice of the resampling distribution $\cq_n$. 
(See \cite{soloff2024bagging} for a more complete discussion of this framework, and for examples of resampling methods beyond bagging and subbagging.)
Here we formalize the conditions that $\cq_n$ needs to satisfy, in order for our main results to be applied. 
\begin{assumption}\label{asm:cq_n}
Fix a sample size $n\geq 1$. The following conditions must be satisfied by the resampling distribution $\cq_n$. \vspace{-1.5em}
\end{assumption}
\emph{
\begin{itemize}
    \item \textbf{Symmetry:} for all $m\geq 1$, $i_1,\ldots,i_m\in [n]$, and permutations $\sigma\in \mathcal{S}_n$,
    \[\cq_n\{(i_1,\ldots,i_m)\} = \cq_n\{(\sigma(i_1),\ldots,\sigma(i_m))\}.\]
    \item \textbf{Nontrivial subsampling:} it holds that $p_{\cq_n}<1$, where we define
    \[p_{\cq_n} = \EE_{r\sim\cq_n}\left[\frac{1}{n}\sum_{i=1}^n \mathbf{1}\{i\in r\}\right],\]
    the expected fraction of data points appearing in a bag $r$ sampled from $\cq_n$.
    \item \textbf{Nonpositive covariance:}
    \[\textnormal{Cov}_{r\sim \cq_n}(\mathbf{1}_{i\in r}, \mathbf{1}_{j\in r})\le 0\textnormal{  for any $i\ne j\in [n]$}.\]
    \item \textbf{Compatibility with the leave-one-out distribution:} the corresponding distribution $\cq_{n-1}$ on $\textnormal{seq}_{[n-1]}$ must equal the distribution of $r\sim \cq_n$ conditional on the event~$n\not\in r$.
\end{itemize}
}
Before proceeding, we comment on the last part of the assumption: why does $\cq_{n-1}$ appear in our assumptions on $\cq_n$? To study the stability of the bagged algorithm (\Cref{def:bag}), it is not sufficient to consider only $\cq_n$, the distribution on $\seqn$---when we compare to the result of the algorithm on the leave-one-out data set $\cd^{\setminus i}$, we now have sample size $n-1$ (instead of $n$), and consequently we need to specify a distribution $\cq_{n-1}$ on $\textnormal{seq}_{[n-1]}$, as well.

The four conditions of \Cref{asm:banach_finite_approx} are all satisfied by the standard resampling schemes described above, with
$p_{\cq_n} = 1 - \left(1-\frac{1}{n}\right)^m$ for bootstrapping (sampling $m$ out of $n$  points uniformly \emph{with} replacement),
and
$p_{\cq_n} = \frac{m}{n}$ for subbagging (sampling $m$ out of $n$  points uniformly \emph{without} replacement). 

\section{Main results for a Hilbert space}\label{sec:hilbert_theory}
We are now ready to present our stability guarantees for bagging. 
Our first main result establishes that derandomized bagging automatically provides mean-square stability for any algorithm with outputs in a bounded subset of a Hilbert space. 
In this section, we work with stability with respect to the Hilbert norm~$\|\cdot\|_{\ch}$. 
For instance,\begin{itemize}
\item If $\cw\subset  \R^d$, then the Hilbert norm may be chosen to be the Euclidean norm $\|\cdot\|_2$, or more generally, a Mahalanobis distance $\|w\|_\ch = \sqrt{w^\top S w}$ for a fixed positive definite matrix $S$.
\item If $\cw\subset L_2(\R^d)$, the space of square-integrable functions $\R^d\to \R$, then the Hilbert norm is given by the $L_2$ norm on functions, $\|w\|_\ch = (\int_{\R^d}w(x)^2\;\textnormal{d}x)^{1/2}$. Alternatively, we can take a Sobolev space of functions $\R^d\to\R$ with smoothness parameter $s$ with the Hilbert norm being the Sobolev norm $\|\cdot\|_{s,2}$ (defined formally later on in \eqref{def:sobolev-norm}---see Experiment 3 in \Cref{sec:data} for details).  
\end{itemize}

\begin{theorem}\label{thm:hilbert} Suppose~$\ca$ takes values in a convex and closed subset~$\cw\subset \ch$, where $\ch$ is a separable Hilbert space, and define \[\radius_{\ch}(\cw) = \inf_{w\in\cw}\sup_{w'\in\cw}\|w-w'\|_{\ch}\]
as the radius of the set $\cw$. Fix a resampling distribution~$\cq_n$ on $\seqn$, which satisfies \Cref{asm:cq_n}.
Then derandomized bagging, run with base algorithm~$\ca$ and resampling distribution~$\cq_n$, has  mean-square stability $\beta^2$ with respect to~$\|\cdot\|_\ch$, with
\begin{align}\label{eq-beta}
\beta^2
:= \frac{\radius^2_{\ch}(\cw)}{n-1}\cdot\frac{p_{\cq_n}}{1-p_{\cq_n}}.
\end{align}
\end{theorem}
 
Combining this result with \Cref{prop:stability_defs} immediately yields a guarantee on $(\varepsilon, \delta)$-stability.
\begin{corollary}\label{cor:eps_delta_stability}
Under the notation and assumptions of \Cref{thm:hilbert}, 
derandomized bagging has tail stability $(\varepsilon, \delta)$ with respect to~$\|\cdot\|_\ch$ for any $\eps,\delta\geq 0$ satisfying
\begin{align}\label{eq-eps-delta}
\delta\eps^2\ge \frac{\radius_{\ch}^2(\cw)}{n-1}\cdot \frac{p_{\cq_n}}{1-p_{\cq_n}}.
\end{align}
\end{corollary}

\paragraph{Comparing to prior bounds}
The main result of \cite[Theorem 8]{soloff2024bagging}, ensuring the stability of algorithms that return a real-valued output, can be recovered as a special case of \Cref{cor:eps_delta_stability} above, by taking $\cw = [0,1]\subseteq \ch = \R$ (note that we then have $\radius_{\ch}(\cw) = \frac{1}{2}$). But remarkably, the result of \Cref{thm:hilbert} shows that there is no price to pay for high-dimensionality (or even for infinite dimensionality): the stability guarantee is \emph{exactly the same} for any Hilbert space, regardless of dimension, as long as the radius of the output space $\cw$ is bounded.

\subsection{Results for a finite number of bags}
The results above apply to the derandomized bagged version of $\ca$, as constructed in \Cref{def:derand-bag}. 
In practice, of course, it is not feasible to calculate this output $\hat{w}$ since we would need to average over all possible bags $r$, so we would instead use the bagged estimator $\hat{w}$ constructed in \Cref{def:bag} for some finite number of bags $B$. Since we are simply taking a Monte Carlo approximation to the derandomized bagged estimator of \Cref{def:derand-bag}, we would expect to incur error on the order of $\frac{1}{\sqrt{B}}$. 
To formalize this, we first recall a generalization of the 
Azuma-Hoeffding inequality:

\begin{theorem}[\citep{hayes2005large}] \label{thm:hayes}  Let $\hat{w}_1,\ldots,\hat{w}_B\in \ch$ be independent random variables in a Hilbert space~$\ch$ such that $\|\hat{w}_b - \mathbb{E}[\hat{w}_b]\|_\ch\le C$ almost surely for $b=1,\ldots,B$. Then for any $\delta > 0$,
\[
\mathbb{P}\left\{\bigg\|\frac{1}{B}\sum_{b=1}^B(\hat{w}_b - \mathbb{E}[\hat{w}_b])\bigg\|_\ch \le C\sqrt{\frac{1}{2B}\log\left(\frac{2e^2}{\delta}\right)}\right\} \ge 1-\delta.
\]
In particular, 
\[
\E\left[\bigg\|\frac{1}{B}\sum_{b=1}^B(\hat{w}_b - \mathbb{E}[\hat{w}_b])\bigg\|_\ch^2\right]\le \frac{C^2e^2}{B}.
\]
\end{theorem}
Note that, aside from the additional factor $e^2$, this result is the same as what we would obtain for the case $\ch = \R$. Combining this result with \Cref{thm:hilbert} yields the following stability guarantee for bagged algorithms with a finite number of bags $B$.

\begin{corollary}\label{cor:hilbert_finite_B}
Under the notation and assumptions of \Cref{thm:hilbert}, let $B\geq 1$ be fixed.
Then bagging, run with base algorithm~$\ca$ and resampling distribution~$\cq_n$ and with $B$ bags, has  mean-square stability $\beta^2$ with respect to~$\|\cdot\|_\ch$, with
\[
\beta^2
:= \radius^2_{\ch}(\cw)\left(\frac{1}{n-1}\cdot\frac{p_{\cq_n}}{1-p_{\cq_n}} + \frac{16e^2}{B}\right).
\]
\end{corollary}
\noindent Of course, tail stability follows as an immediate consequence of \Cref{prop:stability_defs}. 

\section{Extension: results outside the Hilbert space setting}\label{sec:banach_theory}
We now turn to the question of stability in a more general setting, where the output space is not necessarily a Hilbert space. We assume that $\ca$ returns outputs lying in $\cw\subset\cb$, where $\cb$ is a separable Banach space equipped with norm $\|\cdot\|$, and we now aim to establish stability with respect to this norm $\|\cdot\|$ directly.

As a motivating example, consider a setting where~$\ca$ returns an output $\hat{w}$ that specifies a distribution---for instance, this may arise in Bayesian inference (Problem Setting 2 from \Cref{sec:problem_settings}), where $\ca$ returns a posterior distribution over some parameter of interest. 
In this type of setting, 
since $\hat{w}$ represents a distribution, it is natural to study its stability with respect to $\dtv$, the total variation distance.
While $\dtv$ defines a valid norm, this is not a Hilbert norm and so the theoretical guarantees of \Cref{sec:hilbert_theory} cannot be directly applied.

To allow for meaningful stability guarantees in a broader range of settings, then, we now consider stability in this more general context, where $\|\cdot\|$ is not necessarily a Hilbert norm.

\subsection{Theoretical guarantee}
Unlike the special case of a Hilbert space, in this more general setting we need an additional assumption to ensure that stability will hold.
\begin{assumption}\label{asm:banach_finite_approx}
There exists a constant $R>0$ such that $\sup_{w,w'\in\cw}\|w-w'\|\leq R$, and moreover,
there exist some fixed elements $v_1,\dots,v_K\in \cb$ with unit norm, $\|v_1\|=\dots=\|v_K\|=1$, and a constant $\rho > 0$ such that for any $w,w'\in\cw$, the difference $w-w'$ can be approximated by a linear combination of the $v_j$'s,
    \[\Big\| (w - w') - \sum_{i=1}^K t_i v_i \Big\| \leq \rho\]
    for some $t=(t_1,\dots,t_K)\in\R^K$ with $\|t\|_1\leq R$.
\end{assumption}
\noindent Here $R$ is playing a role analogous to the radius, $\radius_\ch(\cw)$, in the Hilbert space setting. For instance, returning to our earlier example, if $\cw = \Delta_{K-1}$ (and the norm is given by $\dtv$) then we can take $v_j$ to be the $j$th canonical basis vector. Then \Cref{asm:banach_finite_approx} holds, with this same value of $K$, and with $R=1$ and $\rho=0$.

\begin{theorem}\label{thm:banach}
Suppose~$\ca$ takes values in a convex and closed  subset~$\cw\subset \cb$, where $\cb$ is a separable Banach space equipped with norm $\|\cdot\|$, and where $\cw$ satisfies \Cref{asm:banach_finite_approx} for some integer $K\geq 1$ and some $R,\rho\geq 0$. Fix a resampling distribution~$\cq_n$ on $\seqn$, which satisfies \Cref{asm:cq_n}.
Then derandomized bagging, run with base algorithm~$\ca$ and resampling distribution~$\cq_n$, has mean-square stability $\beta^2$ with respect to~$\|\cdot\|$, with
\[
\beta^2
:= 1.6 \left[R\sqrt{\frac{2(1+\eta_n)\log(2K)}{n}\left(\frac{p_{\cq_n}}{1-p_{\cq_n}}+\frac{4\eta_n}{(1-p_{\cq_n})^2}\right)} + R\sqrt{\frac{1}{n} \cdot \frac{p_{\cq_n}}{1-p_{\cq_n}}} + 2\rho\cdot p_{\cq_n}\right]^2,\]
where  $\eta_n = \frac{H_n-1}{n-H_n}\asymp \frac{\log n}{n}$, for $H_n = 1 + \frac{1}{2}+\dots+\frac{1}{n}$ denoting the $n$th harmonic number.
\end{theorem}

As in the Hilbert space setting, we can also extend this theorem to verify tail stability and to prove a stability result for bagging with a finite $B$---see \Cref{sec:appendix_banach} for details.

\subsection{Examples}
Now we return to Problem Setting 2, as discussed in the motivation for the extension to Banach spaces, where $\ca$ returns a distribution (e.g., a posterior), to see what the new theorem implies in this setting.
To make this more concrete, we can consider two specific scenarios: a discrete setting and a continuous setting. We will see how, in each scenario, the new result of \Cref{thm:banach} gives a much stronger guarantee than what can be derived from the results of \Cref{sec:hilbert_theory}. For intuition, in this discussion, we treat $\frac{p_{\cq_n}}{1-p_{\cq_n}}$ as a constant.

\paragraph{Problem Setting 2(a): the discrete case}
First, consider a finite setting, where  $\ca$ returns a distribution on $K$ many values---for instance, if the output of $\ca$ specifies a Bayesian posterior over $K$ many possible models. In this scenario, $\hat{w}$ takes values in the simplex $\Delta_{K-1}\subset\R^K$, and so the total variation distance is given by $\dtv(w,w') = \frac{1}{2}\|w-w'\|_1$. 

Of course, since $\R^K$ is itself a Hilbert space (with norm $\|\cdot\|_{\ch} = \|\cdot\|_2$), we can apply our results from \Cref{sec:hilbert_theory}---but this will not give a very useful bound for the stability of derandomized bagging. Specifically, \Cref{thm:hilbert} tells us that  mean-square stability with respect to $\|\cdot\|_2$ holds, for $\beta^2\propto 1/n$. Since $\|\cdot\|_1\leq \sqrt{K}\|\cdot\|_2$ (the usual inequality relating $\ell_1$ norm to the $\ell_2$ norm in $\R^K$), this immediately implies that
\[
\frac{1}{n}\sum_{i\in[n]}\EE\left[\dtv^2\left(\hat{w}, \hat{w}^{\setminus i}\right)\right] \leq \frac{K}{4} \cdot \frac{1}{n}\sum_{i\in[n]}\EE[\|\hat{w} - \hat{w}^{\setminus i}\|^2_2] \lesssim K/n,\]
or in other words,  mean-square stability with respect to $\dtv$ holds, for $\beta^2\propto K/n$. But if the dimension $K$ is large relative to $n$, then this is not a very useful guarantee.

In contrast, let us now see what can be guaranteed via \Cref{thm:banach}. As mentioned above, \Cref{asm:banach_finite_approx} holds with $K$ equal to the dimension, and with $R=1$ and $\rho=0$. Then \Cref{thm:banach} tells us that mean-square stability with respect to $\|\cdot\|$ holds for $\beta^2\propto \frac{\log K}{n}$---a far better bound.

\paragraph{Problem Setting 2(b): the continuous case} Next, consider a continuous setting, where $\ca$ returns a density on $[0,1]$, given by some function $\hat{w}:[0,1]\rightarrow[0,\infty)$ satisfying $\int_{t=0}^1\hat{w}(t)\;\textnormal{d}t=1$. 
In this case we have $\dtv(w,w') = \frac{1}{2}\int_{t=0}^1 |w(t) - w'(t)|\;\textnormal{d}t$.

To formalize this setting, consider the Banach space 
\[\cb = L_1([0,1]) := \left\{f:[0,1]\to\R: \int_{t=0}^1|f(t)|\;\textnormal{d}t<\infty\right\},\] equipped with the norm $\|f\|= \frac{1}{2}\int_{t=0}^1 |f(t)|\;\textnormal{d}t$ (note that this differs from the usual $\ell_1$ norm by the factor~$\frac{1}{2}$, in order to agree with total variation distance). Suppose that our algorithm returns densities that are further constrained to be $L$-Lipschitz, so that we have
\[\cw = \left\{w:[0,1]\to[0,\infty) : \int_{t=0}^1 w(t) = 1, \ \sup_{t\neq t'}\frac{|w(t)-w(t')|}{|t-t'|}\leq L\right\}. \]

As for Problem Setting 2(a), we again begin by asking whether the Hilbert space result of \Cref{thm:hilbert} might already be sufficient in this setting. Indeed, while $\cb = L_1([0,1])$ is not a Hilbert space, the Lipschitz constraint on $\cw$ means that $\int_{t=0}^1 |w(t)|^2\;\textnormal{d}t<\infty$ for all $w\in\cw$, meaning that $\cw$ can be viewed as a subset of the Hilbert space $L_2([0,1])$, equipped with norm $\|f\|_\ch = \|f\|_{L_2} = (\int_{t=0}^1 |f(t)|^2\;\textnormal{d}t)^{1/2}$. 
A straightforward calculation verifies that $\radius_\ch(\cw)\propto L^{1/4}$,
meaning that by \Cref{thm:hilbert}, derandomized bagging satisfies  mean-square stability with respect to $\|\cdot\|_{L_2}$ for $\beta^2\propto L^{1/2}/n$. Since for this function space we have $\|\cdot\|_{L_1}\leq \|\cdot\|_{L_2}$, this immediately implies that  mean-square stability with respect to $\dtv$ also holds with parameter $\beta^2\propto L^{1/2}/n$. However, if the Lipschitz constant $L$ is large, this may not be a satisfying result.

Now we turn to our new result, \Cref{thm:banach}, to see if the stability guarantee can be improved. First, observe that \Cref{asm:banach_finite_approx} holds for any $K\geq 1$ if we take $R=1$ and $\rho \propto L/K$. This holds because we can take $v_j$ to be the function $v_j(t) = K \cdot \mathbf{1}\{\frac{j-1}{K} \leq t \leq \frac{j}{K}\}$, for each $j\in[K]$---then for any $L$-Lipschitz densities $w,w'$, the $2L$-Lipschitz difference $w-w'$ can be approximated with a piecewise constant function, i.e., a linear combination of these $v_j$'s. Choosing $K\propto Ln$, 
we apply \Cref{thm:banach}, which then leads to a guarantee on the  mean-square stability with respect to $\dtv$, with $\beta^2\propto \frac{\log(Ln)}{n}$. Again, as for the discrete case, if $L$ is large (i.e., $L\gg \log^2 n$) then this is a much stronger bound than the one we can obtain via the Hilbert space result. 

\paragraph{Summary of examples}
The following table summarizes our results for the two settings, in terms of the stability guarantee with respect to $\dtv$. The improved scaling in the right column highlights the utility of \Cref{thm:banach} in extending our results to a broader range of settings.

\begin{center}
\begin{tabular}{c|c|c}
    Output space $\cw$ & \begin{tabular}{@{}c@{}}Stability guarantee for $\dtv$\\ via \Cref{thm:hilbert}\\ (Hilbert space)\end{tabular} & \begin{tabular}{@{}c@{}}Stability guarantee for $\dtv$\\ via \Cref{thm:banach}\\ (Banach space)\end{tabular}\\\hline \hline
    \begin{tabular}{@{}c@{}}Problem Setting 2(a):\\
    distributions on $\{1,\dots,K\}$\end{tabular} & $\beta^2\propto \frac{K}{n}$& $\beta^2\propto \frac{\log K}{n}$\\\hline
    \begin{tabular}{@{}c@{}}Problem Setting 2(b):\\$L$-Lipschitz densities on $[0,1]$\end{tabular} &$\beta^2\propto \frac{L^{1/2}}{n}$ & $\beta^2\propto \frac{\log(Ln)}{n}$\\
\end{tabular}
\end{center}
~

\begin{remark}[On the necessity of \Cref{asm:banach_finite_approx}]\label{rmk:banach_finite_approx}
    Our stability guarantee for general Banach spaces essentially requires that $\cw$ lies (approximately) in a convex hull of $2K$ many points---namely, $Rv_1,-Rv_1$, \ldots, $Rv_K,-Rv_K$. In particular, this implies that $\cw$ is (approximately) contained in a \emph{finite-dimensional} subspace of $\cb$.  
    In contrast, our Hilbert space results apply regardless of dimensionality---while we do require $\cw$ to be bounded via its radius, $\radius_\ch(\cw)$, there is no implicit cost to the dimension of $\ch$ in the results of \Cref{thm:hilbert}.
    It turns out that some assumption of this form is actually required for the Banach case---see \Cref{prop:banach_counterexample} (in \Cref{sec:finite-dim-necessary}) for an explicit counterexample verifying stability fails without this type of assumption.
\end{remark}

\section{Experiments}\label{sec:apps}

In this section, we study the stability of subbagging in experiments across several simulation and real data settings.
Code to reproduce all experiments is available at 

\begin{center}
    \url{https://github.com/jake-soloff/stability-resampling-experiments}.
\end{center}

\subsection{Data and Methods}
\label{sec:data}
We start by presenting an overview of our four experiments. Experiments 1, 2, and 3 each examine stability in a Hilbert space, with respect to the appropriate Hilbert norm $\|\cdot\|_\ch$ (as studied in \Cref{sec:hilbert_theory}). In contrast, the setting of Experiment 4 is a Banach space (as in \Cref{sec:banach_theory}). 
    \paragraph{Experiment 1 (Estimating a regression function using bagged decision trees)} We draw data from the following data-generating process:
\begin{align*}
    (X_i,\alpha_i,\gamma_i) &\simiid \textnormal{Uniform}([0, 1]^d)\times \textnormal{Uniform}([-.25, .25])\times \textnormal{Uniform}([0, 1]),\\
    \tilde{Y}_i &= \sum_{j=1}^d \sin\left(\frac{X_{ij}}{j}\right) + \alpha_i \mathbf{1}\left\{i = 1\Mod 3\right\} + \gamma_i \mathbf{1}\left\{i = 1\Mod 4\right\},
\end{align*}
with $n=500$ and~$d = 40$. We then set $Y_i = \frac{\tilde{Y}_i - \min_j \tilde{Y}_j}{\max_j \tilde{Y}_j - \min_j\tilde{Y}_j}$ so that each response $Y_i$ is in the unit interval $[0,1]$. Note that the algorithm has access to the observed data~$\mathcal{D} = (X_i,Y_i)_{i\in [n]}$, i.e., $\alpha_i$ and $\gamma_i$ are latent variables used only to generate the data~$\mathcal{D}$. We apply \verb+sklearn.tree.DecisionTreeRegressor+ to train the regression trees, setting \verb+max_depth=50+ and leaving all other parameters at their default values. We evaluate the stability of the learned function in the $L_2$ norm $\|\hat{w}\|_{L_2} = (\int_{[0,1]^d}\hat{w}(x)^2\textnormal{d}x)^{1/2}$.
\paragraph{Experiment 2 (Synthetic controls real data example)} For this experiment, we reproduce  \cite{abadie2015comparative}’s SCM analysis on the 1990 reunification of Germany, as discussed above in \Cref{sec:problem_settings} (see Problem Setting 3).
The data contains $n=16$ countries (aside from Germany), with $6$ predictors measured for each country (pre-treatment averages of various macroeconomic factors such as inflation and trade openness---see Table 1 of \cite{abadie2015comparative} for more details). SCM then estimates a weighted combination of these $n$ countries that represents a `synthetic West Germany’, meaning that the output of the algorithm is a vector $\hat{w}\in\Delta_{n-1}$. When we run SCM on a subset of the $n=16$ countries, we assign a weight of $0$ to each omitted country. We run SCM using the Python package \texttt{pysyncon} and make extensive use of their reproduction of the original study \cite{abadie2015comparative}. We evaluate the stability of these weights in the Euclidean norm~$\|\cdot\|_2$. 

\paragraph{Experiment 3 (Least-squares spectrum analysis)} We study function estimation in a Sobolev space, corresponding to the Hilbert space $W^{s,2}(\R)$ with smoothness parameter $s=2$, 
with the corresponding norm
\begin{align}\label{def:sobolev-norm}
    \|f\|_{s,2} := \sqrt{\sum_{k=-\infty}^\infty (1+k^2)^s\cdot |\hat{f}(k)|^2},
\end{align}
where $\hat{f}(k)$ denotes the $k$th Fourier coefficient, indexed by $k\in \mathbb{Z}$. We wish to recover $f$, which is equivalent to estimating the Fourier coefficients. We restrict our estimation problem to \emph{even functions}, so that we may constrain the Fourier coefficients to be real valued.

Least squares spectral analysis (LSSA) is one method for estimating the Fourier coefficients of a function from irregularly spaced samples \cite{vanivcek1969approximate}. Given a data set $\cd = (X_i, Y_i)_{i\in [n]}$, we first construct a countable set of features $\Psi_{ik} = \exp\left(jkX_i\right)$, where $j = \sqrt{-1}$, $i\in [n]$ and $k\in \mathbb{Z}$. The base algorithm $\ca(\cd)$ returns a function $\hat{w} : [-1,1]\to \R$, given by
\[
\hat{w}(x)
= \sum_{k=-\infty}^\infty \hat{a}_k\exp(jkx), 
\qquad\textnormal{with}\qquad
\hat{a} = \argmin_{a\in \mathcal{C}_n(R)} \sum_{i=1}^n\left(Y_i - \sum_{k=-\infty}^\infty\Psi_{ik} a_k\right)^2,
\]
and where the constraint set is given by 
\[
\mathcal{C}_n(R) := \left\{(\ldots, a_{-1}, a_0, a_1,\ldots) : a_k\in \R, a_k = 0 \,\forall\,|k| > \lceil{n/2} \rceil, \sum_{k=-\infty}^\infty (1+k^2)^2a_k^2\le R^2\right\}.
\]

For our simulation, we sample $n=200$ random variables $X_1,\ldots,X_n\in [-1,1]$ iid from a non-uniform density $g_X(x) = -\frac{1}{2}\log|x|$. Then, the corresponding responses are given by $Y_i = f^*(X_i) + \zeta_i$, where $\zeta_i\simiid \mathcal{N}(0,.01)$ and $f^*(x) = \sin\left(\frac{2}{1-x} + \frac{2}{1+x}\right)$. We run the base algorithm above using the Python package $\texttt{cvxpy}$, setting $R=10$. We evaluate the stability of the learned function in the Sobolev norm $\|\cdot\|_{2,2}$ (see \Cref{def:sobolev-norm} above).

\paragraph{Experiment 4 (Stability of the softmax function)} We simulate a binary $n\times d$ data matrix $X\in \{0,1\}^{n\times d}$ using $X_{ij}\simiid \textnormal{Bern}(0.2)$, where $n=2000$ and $d=100$. The base algorithm takes the softmax of the column sums $\mathcal{A}(\cd):= \sigma(\mathbf{1}_n^\top X)\in \Delta_{d-1}$, i.e.
\[
\mathcal{A}(\cd)_j =  \frac{\exp\left(\sum_{i=1}^n X_{ij}\right)}{\sum_{k=1}^d \exp\left(\sum_{i=1}^n X_{ik}\right)}.
\]
We evaluate  stability with respect to the total variation distance $\dtv$. (As discussed in \Cref{sec:banach_theory}, the $\dtv$ norm is not a Hilbert norm---that is, unlike the first three experiments, stability guarantees for Experiment 4 hold only due to the extension to the Banach space setting given in \Cref{thm:banach}.)

\begin{figure}[ht!]
    \centering
    Experiment 1: Regression trees
    \includegraphics{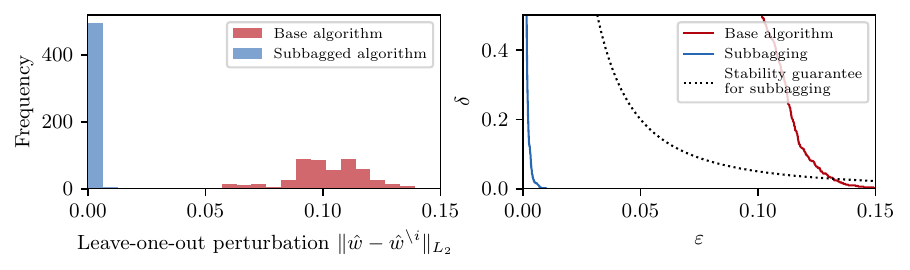}
    Experiment 2: Synthetic controls, German reunification example \citep{abadie2015comparative}
    \includegraphics{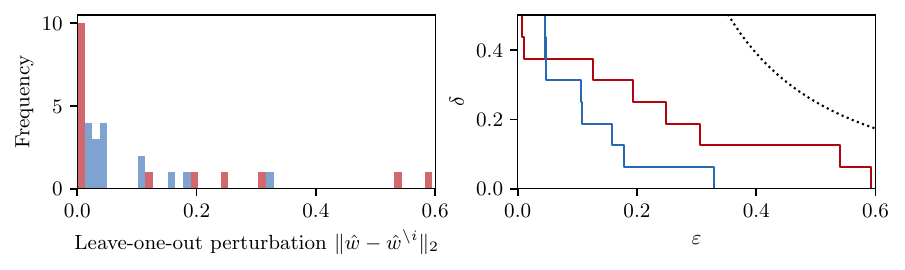}
    Experiment 3: Least-squares spectral analysis
    \includegraphics{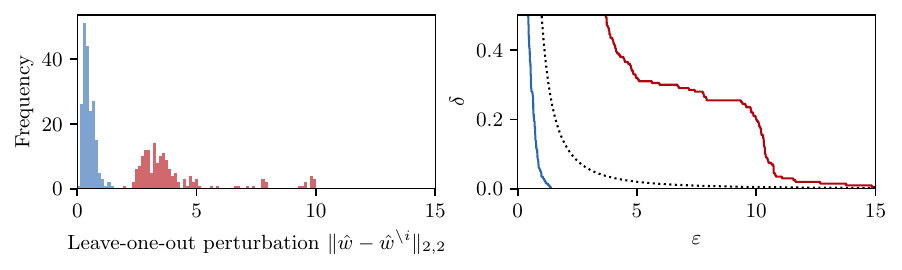}
    Experiment 4: Softmax
    \includegraphics{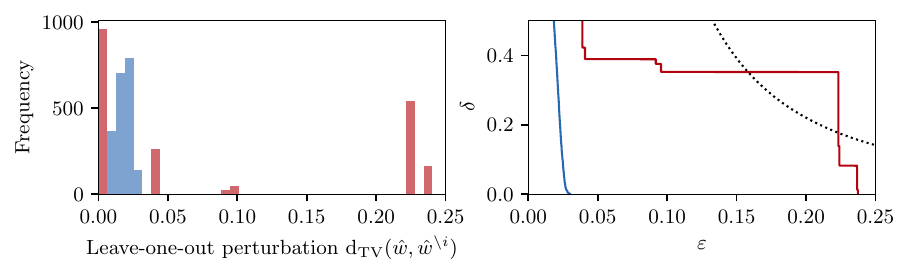}
    \caption{Simulation results comparing the stability of subbagging to that of the corresponding base algorithm under the given norm. Left: Histogram of leave-one-out perturbations. Right: for each $\varepsilon$, the smallest $\delta$ such that the algorithm has tail stability $(\varepsilon, \delta)$ in the sense of~\Cref{def:stability}. Higher curves thus represent greater instability. In the right panels we also plot the stability guarantee of \Cref{thm:hilbert} for Experiments~1--3 and \Cref{thm:banach} for Experiment~4 (the guarantees for derandomized bagging). For all four experiments, $m = n/2$ and $B = 10,000$.}
    \label{fig:experiments-overview}
\end{figure}

\subsection{Results}

Our experiments compare the stability of the base algorithm $\ca$ to the stability of subbagging with $m=n/2$ and $B=10,000$.
Our results are shown in \Cref{fig:experiments-overview} and \Cref{tab:results}. The left panels of \Cref{fig:experiments-overview} show the histogram of leave-one-out perturbations $\left\|\hat{w} - \hat{w}^{\setminus i}\right\|$ for $i\in\left\{1,\ldots,n\right\}$. In the right panels of \Cref{fig:experiments-overview}, for a fixed data set and algorithm, we measure stability by plotting, for each value of $\varepsilon$, the smallest value of $\delta$ such that the algorithm is $\left(\varepsilon,\delta\right)$-stable: 
\[\delta = \frac{1}{n}\sum_{i=1}^n1_{\|\hat{w}-\hat{w}^{\setminus i}\| \ge \varepsilon}.\]
Furthermore, in \Cref{tab:results}, we record the mean-square stability 
\[
\beta^2 := \frac{1}{n}\sum_{i=1}^n\|\hat{w}-\hat{w}^{\setminus i}\|^2.
\]

For all four experiments, we observe that the subbagged algorithm has better stability than the base algorithm, both in terms of the mean-square stability parameter $\beta^2$ and in terms of the tail stability parameters (across all, or nearly all, values of $\varepsilon$).

\begin{table}[ht]\centering
\begin{tabular}{c|c|c}
    Setting & \begin{tabular}{@{}c@{}}Mean-square stability $\beta^2$\\ for base algorithm\end{tabular} & \begin{tabular}{@{}c@{}}Mean-square stability $\beta^2$\\ for subbagged algorithm\end{tabular} \\
    \hline \hline
    Regression trees & $0.0107$ & $0.0000047$ \\
    \hline
    Synthetic controls (SCM) & $0.0532$ & $0.0125$ \\
    \hline
    Spectral analysis (LSSA) & $40.033$ & $0.0793$ \\
    \hline
    Softmax & $0.0186$ & $0.0006$
\end{tabular}
\label{tab:results}
\caption{Results for mean-square stability.}
\end{table}

\section{Conclusion}\label{sec:conclusion}

In supervised learning, algorithmic stability plays a fundamental role in generalization theory \citep{bousquet2002stability,mukherjee2006learning,shalev2010learnability}. More broadly, various forms of  stability are fundamental across applied statistics, in part for their role in establishing reproducibility and interpretability \citep{yu2020veridical,yu2024}. This work provides general, finite sample theoretical guarantees supporting the use of bagging and related techniques as general-purpose stabilizers. Our stability guarantees represent a major generalization of~\cite{soloff2024bagging}---which studied algorithms with outputs on the real line---as our results apply to infinite dimensional settings.

An important remaining open question is how to define and guarantee stability for algorithms with discrete outputs.
In the synthetic controls setting, for example, our theory shows that bagging can stabilize the SCM \emph{weights}. 
However, our results do not imply stability in the resulting \emph{support}---that is, the set of control units that receive nonzero support. 
In a sense, the goal of sparsity is incompatible with stability \citep{xu2011sparse}, so we need a new framework to reconcile these two desiderata. We leave these questions for future work.

\subsection*{Acknowledgements} 

 We gratefully acknowledge the support of the National Science Foundation via grant DMS-2023109 DMS-2023109. R.F.B. and J.A.S. gratefully acknowledge the support of the Office of Naval Research via grant N00014-20-1-2337. R.M.W. gratefully acknowledges the support of NSF DMS-2235451, Simons Foundation MP-TMPS-00005320, and AFOSR FA9550-18-1-0166.

\appendix
\crefalias{section}{appendix}
\section{Proofs for the Hilbert space setting}\label{sec:appendix_hilbert}

\begin{proof}[Proof of \Cref{thm:hilbert}] 
For each $r\in\seqn$, define
$\hat{w}^r = \EE_\xi[\ca(\cd^r;\xi)]$,
the expected output of the randomized algorithm for the data set given by $\cd^r$. From this point on, we write $\EE[\cdot]$ to denote expectation with respect to $r\sim\cq_n$. We then have $\EE[\hat{w}^r]=\hat{w}$ and 
\[\hat{w}^{\setminus i} = \EE\left[\hat{w}^r \mid i\not\in r\right] = \frac{1}{1-p_{\cq_n}} \EE\left[\hat{w}^r\cdot\mathbf{1}\{i\not\in r\}\right]\]
by construction,
and therefore, a straightforward calculation shows that
\begin{equation}\label{eqn:w_hat_i_identity}\hat{w} - \hat{w}^{\setminus i} = \frac{1}{1-p_{\cq_n}} \EE\left[(\hat{w}^r - \hat{w})\cdot\left(\mathbf{1}\{i\in r\} - p_{\cq_n} \right)\right].\end{equation}
Therefore,
\begin{align*}
    \sum_{i\in[n]} \|\hat{w}-\hat{w}^{\setminus i}\|^2_\ch
    &= \sum_{i\in[n]} \left\langle \hat{w}-\hat{w}^{\setminus i}, \frac{1}{1-p_{\cq_n}} \EE\left[(\hat{w}^r-\hat{w})\cdot\left(\mathbf{1}\{i\in r\} - p_{\cq_n}\right)\right] \right\rangle\\
    &= \frac{1}{1-p_{\cq_n}} \EE\left[\left\langle \sum_{i\in[n]} (\hat{w}-\hat{w}^{\setminus i})\cdot\left(\mathbf{1}\{i\in r\} - p_{\cq_n}\right), \hat{w}^r- \hat{w} \right\rangle\right]\\
    &\leq \frac{1}{1-p_{\cq_n}} \EE\left[\left\| \sum_{i\in[n]} (\hat{w}-\hat{w}^{\setminus i})\cdot\left(\mathbf{1}\{i\in r\} - p_{\cq_n}\right)\right\|_{\ch}^2\right]^{1/2}\cdot\EE\left[\left\| \hat{w}^r- \hat{w} \right\|_{\ch}^2\right]^{1/2}\\
    &\leq \frac{\radius_\ch(\cw)}{1-p_{\cq_n}} \EE\left[\left\| \sum_{i\in[n]} (\hat{w}-\hat{w}^{\setminus i})\cdot\left(\mathbf{1}\{i\in r\} - p_{\cq_n}\right)\right\|_{\ch}^2\right]^{1/2}.
\end{align*}
The last step holds because we have
\[\EE\left[\|\hat{w}^r - \hat{w}\|_\ch^2\right]
= \inf_{w\in\cw}\EE\left[\|\hat{w}^r - w\|_\ch^2\right] \leq \inf_{w\in\cw}\sup_{w'\in\cw}\|w'-w\|_{\ch}^2 = \radius^2_{\ch}(\cw),\]
where the first equality holds since $\hat{w} = \EE[\hat{w}^r]$ minimizes mean-square error.

To calculate the last remaining expected value, we first note that since we have assumed symmetry and nonpositive covariance for $\cq_n$ (see \Cref{asm:cq_n}), we have $\textnormal{Cov}\left(\mathbf{1}\{i\in r\},\mathbf{1}\{j\in r\}\right) = -q$ for all $i\neq j$, for some $q\geq0$. Moreover, since the covariance matrix of the random vector $(\mathbf{1}\{i\in r\})_{i\in[n]}$ must be positive semidefinite, we must have $q\leq \frac{p_{\cq_n}(1-p_{\cq_n})}{n-1}$. 
We then have
\begin{align*}
    &\EE\left[\left\| \sum_{i\in[n]} (\hat{w}-\hat{w}^{\setminus i})\cdot\left(\mathbf{1}\{i\in r\} - p_{\cq_n}\right)\right\|_{\ch}^2\right]\\
    &=\sum_{i\in[n]}\textnormal{Var}(\mathbf{1}\{i\in r\}) \cdot \|\hat{w}-\hat{w}^{\setminus i}\|^2_\ch +  \sum_{i\neq j\in[n]}\textnormal{Cov}\left(\mathbf{1}\{i\in r\},\mathbf{1}\{j\in r\}\right) \cdot \langle \hat{w}-\hat{w}^{\setminus i},\hat{w}-\hat{w}^{\setminus j}\rangle\\
    &=\sum_{i\in[n]} p_{\cq_n}(1-p_{\cq_n}) \cdot \|\hat{w}-\hat{w}^{\setminus i}\|^2_{\ch} + \sum_{i\neq j\in[n]} (-q) \cdot \langle \hat{w}-\hat{w}^{\setminus i},\hat{w}-\hat{w}^{\setminus j}\rangle\\
    &=  \left(p_{\cq_n}(1-p_{\cq_n}) + q\right) \cdot \sum_{i\in[n]}\|\hat{w}-\hat{w}^{\setminus i}\|^2_{\ch} - q\cdot \left\|\sum_{i\in[n]}\hat{w}-\hat{w}^{\setminus i}\right\|^2_{\ch}\\
    &\leq \frac{n}{n-1}\cdot p_{\cq_n}(1-p_{\cq_n}) \cdot \sum_{i\in[n]} \|\hat{w}-\hat{w}^{\setminus i}\|^2_{\ch}.
\end{align*}
Combining everything, then, we have
\[\sum_{i\in[n]}\|\hat{w}-\hat{w}^{\setminus i}\|^2_{\ch}\leq \frac{\radius_\ch(\cw)}{1-p_{\cq_n}} \left( \frac{n}{n-1}\cdot p_{\cq_n}(1-p_{\cq_n}) \cdot \sum_{i\in[n]}\|\hat{w}-\hat{w}^{\setminus i}\|^2_{\ch}\right)^{1/2},\]
which yields
\[\frac{1}{n}\sum_{i\in[n]}\|\hat{w}-\hat{w}^{\setminus i}\|^2_{\ch}\leq \frac{\radius^2_\ch(\cw)}{n-1}\cdot\frac{p_{\cq_n}}{1-p_{\cq_n}}\]
after rearranging terms.
\end{proof}

\begin{proof}[Proof of \Cref{cor:hilbert_finite_B}]
Let $\hat{w}$ and $\hat{w}^{\setminus i}$ be the output of the bagged algorithm (as in \Cref{def:bag}), and let $\tilde{w}$ and $\tilde{w}^{\setminus i}$ be the corresponding derandomized bagged estimators (as in \Cref{def:derand-bag}). For each $i\in[n]$, we have
\begin{multline*}\EE\left[\|\hat{w} - \hat{w}^{\setminus i}\|^2_\ch\right]
= \EE\left[\|(\hat{w} - \tilde{w})  - (\hat{w}^{\setminus i} - \tilde{w}^{\setminus i}) + (\tilde{w} - \tilde{w}^{\setminus i})\|^2_\ch\right]\\
= \EE\left[\|(\hat{w} - \tilde{w})  - (\hat{w}^{\setminus i} - \tilde{w}^{\setminus i}) \|^2_\ch \right]+ \|\tilde{w} - \tilde{w}^{\setminus i}\|^2_\ch,\end{multline*}
since $\tilde{w} - \tilde{w}^{\setminus i}$ is nonrandom, while $(\hat{w} - \tilde{w})  - (\hat{w}^{\setminus i} - \tilde{w}^{\setminus i})$ has mean zero. Next,
\[\EE\left[\|(\hat{w} - \tilde{w})  - (\hat{w}^{\setminus i} - \tilde{w}^{\setminus i}) \|^2_\ch \right]
\leq 2\EE\left[\|\hat{w} - \tilde{w}\|^2_\ch\right] + 2\EE\left[\|\hat{w}^{\setminus i} - \tilde{w}^{\setminus i} \|^2_\ch \right].\]
We next have $\hat{w} = \frac{1}{B}\sum_{b=1}^B\ca(\cd^{r_b};\xi_b)$. Note that the random variables $\ca(\cd^{r_b};\xi_b)\in\ch$ are i.i.d., with mean $\tilde{w}$, and with $\|\ca(\cd^{r_b};\xi_b) - \tilde{w}\|_{\ch}\leq 2\radius_{\ch}(\cw)$ almost surely. Applying \Cref{thm:hayes}, 
\[\EE\left[\|\hat{w} - \tilde{w}\|^2_\ch\right] \leq \frac{4e^2 \radius^2_{\ch}(\cw)}{B}.\]
By an identical argument, the same holds for the leave-one-out outputs, $\hat{w}^{\setminus i}$ and $ \tilde{w}^{\setminus i}$. Therefore,
\[\EE\left[\|\hat{w} - \hat{w}^{\setminus i}\|^2_\ch\right] \leq  \|\tilde{w} - \tilde{w}^{\setminus i}\|^2_\ch+ \frac{16e^2 \radius^2_{\ch}(\cw)}{B}.\]
Finally, $\|\tilde{w} - \tilde{w}^{\setminus i}\|^2_\ch$ is bounded by \Cref{thm:hilbert}.
\end{proof}

\section{Proofs and additional results for the Banach space setting}\label{sec:appendix_banach}

\subsection{Proof of main theorem}
\begin{proof}[Proof of \Cref{thm:banach}]
Let $\|\cdot\|_*$ represent the dual norm to $\|\cdot\|$.
First, for each $i$, as a consequence of the Hahn--Banach theorem (see, e.g., \cite{folland1999analysis}, Theorem~5.8), there is a linear operator $\hat{f}_i$ on $\cb$ such that $\|\hat{f}_i\|_*=1$, and \[\hat{f}_i(\hat{w}-\hat{w}^{\setminus i}) = \|\hat{w}-\hat{w}^{\setminus i}\|.\]
We then have
\begin{align*}
    &\sum_{i\in[n]} \|\hat{w}-\hat{w}^{\setminus i}\|^2
    =\sum_{i\in [n]}\|\hat{w}-\hat{w}^{\setminus i}\| \cdot \hat{f}_i(\hat{w}-\hat{w}^{\setminus i})\\
    &=\sum_{i\in[n]}\|\hat{w}-\hat{w}^{\setminus i}\| \cdot \hat{f}_i\left(\frac{1}{1-p_{\cq_n}} \EE\left[(\hat{w}^r - \hat{w})\cdot\left(\mathbf{1}\{i\in r\} - p_{\cq_n} \right)\right]\right)\textnormal{ as calculated in~\eqref{eqn:w_hat_i_identity}}\\
    &=\sum_{i\in[n]}\|\hat{w}-\hat{w}^{\setminus i}\| \cdot \frac{1}{1-p_{\cq_n}} \EE\left[\hat{f}_i(\hat{w}^r - \hat{w})\cdot\left(\mathbf{1}\{i\in r\} - p_{\cq_n} \right)\right]\textnormal{ since $\hat{f}_i$ is linear}\\
    &=\frac{1}{1-p_{\cq_n}}\cdot \EE\left[\sum_{i\in[n]}\|\hat{w}-\hat{w}^{\setminus i}\| \cdot  \hat{f}_i(\hat{w}^r - \hat{w})\cdot\left(\mathbf{1}\{i\in r\} - p_{\cq_n} \right)\right].
\end{align*}

Next, by \Cref{asm:banach_finite_approx}, for each $r\in\seqn $ we can find some $t^r=(t^r_1,\dots,t^r_K)\in\R^K$ with $\|t^r\|_1\leq R$ and
$\Big\|(\hat{w}^r - \hat{w}) -   \sum_{j=1}^K t^r_j v_j\Big\|\leq \rho$.
Then, since $\hat{f}_i$ is linear and has unit dual norm $\|\hat{f}_i\|_*=1$,
$\left|\hat{f}_i(\hat{w}^r - \hat{w}) - \sum_{j=1}^K t^r_j\hat{f}_i(v_j)\right|\leq \rho$
for each $i\in[n]$, and so
\begin{align*}
    \sum_{i\in[n]} \|\hat{w}-\hat{w}^{\setminus i}\|^2
    &=\frac{1}{1-p_{\cq_n}}\cdot \EE\left[\sum_{i\in[n]}\|\hat{w}-\hat{w}^{\setminus i}\| \cdot  \hat{f}_i(\hat{w}^r - \hat{w})\cdot\left(\mathbf{1}\{i\in r\} - p_{\cq_n} \right)\right]\\
    &\leq \frac{1}{1-p_{\cq_n}}\cdot \EE\left[\sum_{j\in[K]} t^r_j \cdot \sum_{i\in[n]}\|\hat{w}-\hat{w}^{\setminus i}\| \cdot  \hat{f}_i(v_j)\cdot\left(\mathbf{1}\{i\in r\} - p_{\cq_n} \right)\right] \\
    &\hspace{1in}+ \frac{\rho}{1-p_{\cq_n}}\cdot \EE\left[\sum_{i\in[n]}\left|\|\hat{w}-\hat{w}^{\setminus i}\| \cdot\left(\mathbf{1}\{i\in r\} - p_{\cq_n} \right)\right|\right]\\
    &\leq \frac{R}{1-p_{\cq_n}}\cdot \EE\left[\max_{j\in[K]}\left|\sum_{i\in[n]}\|\hat{w}-\hat{w}^{\setminus i}\| \cdot  \hat{f}_i(v_j)\cdot\left(\mathbf{1}\{i\in r\} - p_{\cq_n} \right)\right|\right] \\
    &\hspace{1in}+ \frac{\rho}{1-p_{\cq_n}}\cdot \EE\left[\sum_{i\in[n]}\left|\|\hat{w}-\hat{w}^{\setminus i}\| \cdot\left(\mathbf{1}\{i\in r\} - p_{\cq_n} \right)\right|\right],
\end{align*}
where for the last step, we use the fact that $\|t^r\|_1\leq R$ for any $r$. Since $\EE[|\mathbf{1}\{i\in r\} - p_{\cq_n}|] = 2p_{\cq_n}(1-p_{\cq_n})$ for each $i$, and $\sum_{i\in[n]} \|\hat{w}-\hat{w}^{\setminus i}\| \leq \sqrt{n}\cdot \sqrt{\sum_{i\in[n]}\|\hat{w}-\hat{w}^{\setminus i}\|^2}$, we can rewrite this as
\begin{align*}
    \sum_{i\in[n]} \|\hat{w}-\hat{w}^{\setminus i}\|^2
\leq \frac{R}{1-p_{\cq_n}}\cdot &\EE\left[\max_{j\in[K]}\left|\sum_{i\in[n]}\|\hat{w}-\hat{w}^{\setminus i}\| \cdot  \hat{f}_i(v_j)\cdot\left(\mathbf{1}\{i\in r\} - p_{\cq_n} \right)\right|\right] \\
&+ 2\rho\sqrt{n}\cdot p_{\cq_n} \cdot \sqrt{\sum_{i\in[n]}\|\hat{w}-\hat{w}^{\setminus i}\|^2}.
\end{align*}

Next, define
$\hat{p}^r = \frac{1}{n}\sum_{i\in[n]}\mathbf{1}\{i\in r\}$, 
so that $\EE[\hat{p}^r] = p_{\cq_n}$. Using \Cref{asm:cq_n} we can verify that 
$\textnormal{Var}(\hat{p}^r) \leq \frac{p_{\cq_n}(1-p_{\cq_n})}{n}$.
Then, since $|\hat{f}_i(v_j)|\leq 1$ for all $i,j$,
\begin{multline*}
    \EE\left[\max_{j\in[K]}\left|\sum_{i\in[n]}\|\hat{w}-\hat{w}^{\setminus i}\| \cdot  \hat{f}_i(v_j)\cdot\left(\hat{p}^r - p_{\cq_n} \right)\right|\right]
    =\EE\left[\max_{j\in[K]}\left|\sum_{i\in[n]}\|\hat{w}-\hat{w}^{\setminus i}\| \cdot  \hat{f}_i(v_j)\right| \cdot |\hat{p}^r - p_{\cq_n}|\right]\\
    \leq \sum_{i\in[n]}\|\hat{w}-\hat{w}^{\setminus i}\| \cdot \EE\left[|\hat{p}^r - p_{\cq_n}|\right]
    \leq \sqrt{n} \cdot \sqrt{\sum_{i\in[n]}\|\hat{w}-\hat{w}^{\setminus i}\|^2} \cdot \sqrt{\textnormal{Var}(\hat{p}^r)}\\
    \leq \sqrt{p_{\cq_n}(1-p_{\cq_n})} \cdot \sqrt{\sum_{i\in[n]}\|\hat{w}-\hat{w}^{\setminus i}\|^2}.
\end{multline*}

Combining with the work above, then,
\begin{multline*}\sum_{i\in[n]} \|\hat{w}-\hat{w}^{\setminus i}\|^2
\leq \frac{R}{1-p_{\cq_n}}\cdot \EE\left[\max_{j\in[K]}\left|\sum_{i\in[n]}\|\hat{w}-\hat{w}^{\setminus i}\| \cdot  \hat{f}_i(v_j)\cdot\left(\mathbf{1}\{i\in r\} - \hat{p}^r \right)\right|\right] \\
{}+R \sqrt{\frac{p_{\cq_n}}{1-p_{\cq_n}}} \cdot \sqrt{\sum_{i\in[n]}\|\hat{w}-\hat{w}^{\setminus i}\|^2} + 2\rho\sqrt{n}\cdot p_{\cq_n} \cdot \sqrt{\sum_{i\in[n]}\|\hat{w}-\hat{w}^{\setminus i}\|^2}.\end{multline*}

Now we bound the remaining expected value. Let $c^{(j)}\in\R^n$ have entries $c^{(j)}_i = \|\hat{w}-\hat{w}^{\setminus i}\| \cdot \hat{f}_i(v_j)$,  and observe that $\|c^{(j)}\|^2_2 \leq \sum_{i\in[n]} \|\hat{w}-\hat{w}^{\setminus i}\|^2$ and $\|c^{(j)}\|_{\infty} \leq \max_{i\in[n]}\|\hat{w} - \hat{w}^{\setminus i}\|$, for each $j$, because $|\hat{f}_i(v_j)|\leq 1$. Moreover, for any $i$, $\hat{w}-\hat{w}^{\setminus i} = p_{\cq_n}(\hat{w}^i - \hat{w}^{\setminus i})$, where $\hat{w}^i = \EE[\hat{w}^r\mid i\in r]$, and so $\|\hat{w}-\hat{w}^{\setminus i}\|\leq p_{\cq_r} R$ (by \Cref{asm:banach_finite_approx}), and so $\|c^{(j)}\|_{\infty} \leq p_{\cq_n}R$ for all $j$.
Applying \Cref{lem:Bernoulli_exch} (with vectors $c^{(1)},\dots,c^{(K)},-c^{(1)},\dots,-c^{(K)}$), then,
\begin{multline*}\EE\left[ \max_{j\in[K]}\left|\sum_{i\in[n]}\|\hat{w}-\hat{w}^{\setminus i}\| \cdot  \hat{f}_i(v_j)\cdot\left(\mathbf{1}\{i\in r\} - \hat{p}^r \right)\right| \, \middle| \, \hat{p}^r\right]\\\leq \sqrt{\sum_{i\in[n]} \|\hat{w}-\hat{w}^{\setminus i}\|^2}\cdot \sqrt{2(1+\eta_n)(\hat{p}^r(1-\hat{p}^r)+4\eta_n)\log(2K)} + \frac{2}{3}(1+\eta_n) \log(2K)\cdot p_{\cq_n}R.\end{multline*}
Finally, since $t\mapsto \sqrt{t(1-t)+c}$ is concave for $t\in[0,1]$ and for fixed $c\geq 0$, marginalizing over $\hat{p}^r$, and applying Jensen's inequality, yields
\begin{align*}
\EE&\left[ \max_{j\in[K]}\left|\sum_{i\in[n]}\|\hat{w}-\hat{w}^{\setminus i}\| \cdot  \hat{f}_i(v_j)\cdot\left(\mathbf{1}\{i\in r\} - \hat{p}^r \right)\right|\right]\\
&\leq \sqrt{\sum_{i\in[n]} \|\hat{w}-\hat{w}^{\setminus i}\|^2}\cdot \sqrt{2(1+\eta_n)(p_{\cq_n}(1-p_{\cq_n})+4\eta_n)\log(2K)}+ \frac{2}{3}(1+\eta_n) \log(2K)\cdot p_{\cq_n}R. 
\end{align*}

Combining all the above calculations proves that
$\sum_{i\in[n]}\|\hat{w}-\hat{w}^{\setminus i}\|^2 \leq \sqrt{\sum_{i\in[n]}\|\hat{w}-\hat{w}^{\setminus i}\|^2} \cdot A + B $,
where
\[A = R\left[\sqrt{2(1+\eta_n)\left(\frac{p_{\cq_n}}{1-p_{\cq_n}}+\frac{4\eta_n}{(1-p_{\cq_n})^2}\right)\log(2K)} + \sqrt{\frac{p_{\cq_n}}{1-p_{\cq_n}}}\right] + 2\rho\sqrt{n}\cdot p_{\cq_n}\]
and 
\[B = R^2 \frac{p_{\cq_n}}{1-p_{\cq_n}} \cdot \frac{2}{3}(1+\eta_n)\log(2K) \leq \frac{A^2}{3}.\]
Then, solving the quadratic,
\[\sum_{i\in[n]}\|\hat{w}-\hat{w}^{\setminus i}\|^2 \leq \sqrt{\sum_{i\in[n]}\|\hat{w}-\hat{w}^{\setminus i}\|^2} \cdot A + \frac{A^2}{3} \ \Longrightarrow \ \sum_{i\in[n]}\|\hat{w}-\hat{w}^{\setminus i}\|^2 \leq  A^2 \cdot \left(\frac{1 + \sqrt{7/3}}{2}\right)^2\leq 1.6A^2.\]
This completes the proof of the theorem.
\end{proof}

\begin{lemma}\label{lem:Bernoulli_exch}
    Let $B\in\{0,1\}^n$ be a uniformly random permutation of $(\mathbf{1}_k,\mathbf{0}_{n-k})$, and let $p=k/n$. Let $c^{(1)},\dots,c^{(K)}\in\R^n$ be fixed, with $\max_{j\in[K]}\|c^{(j)}\|_2\leq C_2$ and $\max_{j\in[K]}\|c^{(j)}\|_\infty\leq C_\infty$. Then
    \[\EE\left[\max_{j\in[K]} c^{(j)\top} (B - p\mathbf{1}_n)\right] \leq  C_2\sqrt{2(1+\eta_n)(p(1-p)+4\eta_n)\log K} + \frac{2C_\infty}{3}(1+\eta_n) \log K,\]
    where $\eta_n = \frac{H_n - 1}{n-H_n}$, for $H_n = 1+\frac{1}{2}+\dots + \frac{1}{n}$ denoting the $n$th harmonic number.
\end{lemma}
\begin{proof}
        Applying
   \cite[Theorem 10]{barber2024hoeffding}, for each $j$, for any $0<\lambda< \frac{3}{2\|c^{(j)}\|_\infty(1+\eta_n)}$, it holds that \[\EE\left[\exp\left\{\lambda c^{(j)\top} (B - p\mathbf{1}_n)\right\} \right] \leq \exp\left\{ \frac{\lambda^2\|c^{(j)}\|^2_2 (1+\eta_n)}{2(1-\frac{2\lambda}{3}\|c^{(j)}\|_\infty (1+\eta_n))}\cdot \left(p(1-p) + 4\eta_n\right)\right\}.\]
 Therefore, from the definition of $C_2$ and $C_\infty$, for any $0<\lambda<\frac{3}{2(1+\eta_n)C_\infty}$, this implies
\[\EE\left[\exp\left\{\lambda c^{(j)\top} (B - p\mathbf{1}_n)\right\} \right] \leq \exp\left\{ \frac{\lambda^2C_2^2 (1+\eta_n)}{2(1-\frac{2\lambda}{3} (1+\eta_n)C_\infty)}\cdot \left(p(1-p) + 4\eta_n\right)\right\}\]
for each $j$. We then have
\begin{align*}
    \EE\left[\max_{j\in[K]} c^{(j)\top} (B - p\mathbf{1}_n)\right]
    &\leq \lambda^{-1}\log\left(\EE\left[\exp\left\{\lambda \max_{j\in[K]} c^{(j)\top} (B - p\mathbf{1}_n)\right\}\right]\right)\textnormal{ by Jensen's inequality}\\
    &\leq \lambda^{-1}\log\left(\sum_{j\in[K]}\EE\left[\exp\left\{\lambda  c^{(j)\top} (B - p\mathbf{1}_n)\right\}\right]\right)\\
    &\leq \lambda^{-1}\log\left(K \cdot \exp\left\{ \frac{\lambda^2C_2^2 (1+\eta_n)}{2(1-\frac{2\lambda}{3} (1+\eta_n)C_\infty)}\cdot \left(p(1-p) + 4\eta_n\right)\right\}\right)\\
    &=\frac{\log K}{\lambda} + \frac{\lambda C_2^2 (1+\eta_n)}{2(1-\frac{2\lambda}{3} (1+\eta_n)C_\infty)}\cdot \left(p(1-p) + 4\eta_n\right).
\end{align*}
Taking 
\[\lambda = \frac{\sqrt{\frac{2(1+\eta_n)\log K}{C_2^2(p(1-p)+4\eta_n)}}}{(1+\eta_n)\left(1 + \frac{2C_\infty}{3}\sqrt{\frac{2(1+\eta_n)\log K}{C_2^2(p(1-p)+4\eta_n)}}\right)} < \frac{3}{2(1+\eta_n)C_\infty},\]
this yields the desired bound.
\end{proof}

\subsection{Is \texorpdfstring{\Cref{asm:banach_finite_approx}}{Assumption 2} needed?}\label{sec:finite-dim-necessary}

We return to the role of \Cref{asm:banach_finite_approx}, as discussed in \Cref{rmk:banach_finite_approx}. In particular, this assumption requires that $\cw$ is approximately contained in a $K$-dimensional subspace of $\cb$.

We will now see that no uniform stability guarantee can hold without this type of condition. A counterexample can be found with a standard infinite-dimensional Banach space: $\cb = L_1(\mathbb{N}) = \{(x_1,x_2,\dots) : \sum_{j\geq 1} |x_j| < \infty\}$. We will equip $\cb$ with the norm induced by the $\dtv$ metric, given by $\|x\| = \frac{1}{2}\|x\|_1 = \frac{1}{2}\sum_{j\geq 1}|x_j|$.
This is an infinite-dimensional and separable Banach space. 

Next, define $
\cw = \Delta_\infty = \{x\in\cb: x_i\geq 0\ \forall i, \ \|x\|_1 = 1\}$, the infinite-dimensional probability simplex.
To relate this to our earlier examples, this can be viewed as an infinite-dimensional version of Problem Setting 2(b)---$\cw$ is the space of distributions over $\mathbb{N}$, so for instance, we might have an algorithm $\ca$ that returns a posterior distribution over  \emph{countably infinitely} many possible models.
Although $\cw$ is bounded in the norm $\|\cdot\|$, it cannot be approximated by any finite-dimensional subspace in $\cb$. The following result shows stability can fail when bagging algorithms with output in $\cw$.
\begin{prop}\label{prop:banach_counterexample}
   Let $\cz = \mathbb{N}$, the set of positive integers. Let $\cb=L_1(\mathbb{N})$ (equipped with the norm $\|x\|=\frac{1}{2}\sum_{i\geq 1}|x_i|$), and let $\cw=\Delta_\infty$. 
   Then there exists a (deterministic) algorithm $\ca: \cup_{n\geq 0}\cz^n \times [0,1]\to \cw$, such that for all $n\geq 1$, for any distribution $\cq_n$ satisfying \Cref{asm:cq_n}, derandomized bagging fails to satisfy mean-square stability with respect to $\dtv$ for any $\beta < p_{\cq_n}$.
\end{prop}
\noindent In other words, stability results of the type shown in \Cref{thm:banach} guarantee a bound on $\beta^2$ (the stability parameter), uniformly over all algorithms $\ca$, that vanishes as $n\to\infty$. In contrast, the above example, we show that---for a particular construction of the algorithm $\ca$---the stability parameter $\beta^2$ is bounded away from zero even as $n\to\infty$.

\begin{proof}[Proof of \Cref{prop:banach_counterexample}]
Let $\mathcal{S}$ denote all finite subsets of $\mathbb{N}$, and let $\psi:\mathcal{S}\to\mathbb{N}$ be a bijective map that assigns each finite subset of $\mathbb{N}$ to a (unique) positive integer.
For any dataset $\cd = (Z_1,\dots,Z_m)\in\mathbb{N}^m$, we define
  \[S(\cd) = \left\{i\in\mathbb{N}:\sum_{j\in[m]}\mathbf{1}\{Z_j=i\}\geq 1\right\},\]
  the set of unique values appearing in the data set. Then define 
  $\ca(\cd;\xi) = \mathbf{e}_{\psi(S(\cd))}\in\cw,$
  where $\mathbf{e}_j$ is the $j$th canonical basis vector, with a $1$ in the $j$th location and $0$'s elsewhere.
  
  Now consider data set $\cd = (1,\dots,n)$. Let $\hat{w}$ and $\hat{w}^{\setminus i}$ denote the derandomized bagged estimators, as before, for data sets $\cd$ and $\cd^{\setminus i}$, respectively. 
  By construction of $\ca$, we can verify that
  $\hat{w}_j = \PP_{r\sim\cq_n}\{\psi(S(\cd^r))=j\}$
  and similarly 
  $\hat{w}^{\setminus i}_j = \PP_{r\sim\cq_n}\{\psi(S(\cd^r))=j \mid i\not\in r\}$.
Therefore, 
\begin{align*}
    \|\hat{w}-\hat{w}^{\setminus i}\|
    &=\frac{1}{2}\sum_{j\in\mathbb{N}} \left|\hat{w}_j - \hat{w}^{\setminus i}_j\right|
    =\frac{1}{2}\sum_{j\in\mathbb{N}}\left| \PP_{r\sim\cq_n}\{\psi(S(\cd^r))=j\} - \PP_{r\sim\cq_n}\{\psi(S(\cd^r))=j\mid i\not\in r\}\right|\\
    &=\frac{1}{2}\sum_{j\in\mathbb{N}}\left| \PP_{r\sim\cq_n}\{\psi(S(\cd^r))=j\} - \frac{\PP_{r\sim\cq_n}\{\psi(S(\cd^r))=j\}}{1-p_{\cq_n}}\cdot \mathbf{1}\{i\not\in \psi^{-1}(j)\}\right|\\
    &=\frac{1}{2}\sum_{j\in\mathbb{N}}\PP_{r\sim\cq_n}\{\psi(S(\cd^r))=j\}\cdot \left| 1 - \frac{\mathbf{1}\{i\not\in \psi^{-1}(j)\}}{1-p_{\cq_n}}\right|\\
    &=\frac{1}{2}\sum_{j\in\mathbb{N}}\PP_{r\sim\cq_n}\{\psi(S(\cd^r))=j\}\cdot\frac{ (1-p_{\cq_n})\cdot \mathbf{1}\{i\in S(\cd^r)\} + p_{\cq_n}\cdot \mathbf{1}\{i\not\in S(\cd^r)\}}{1-p_{\cq_n}}\\
    &=\frac{1}{2}\frac{1}{1-p_{\cq_n}}\left((1-p_{\cq_n})\cdot \PP_{r\sim\cq_n}\{i\in S(\cd^r)\} + p_{\cq_n}\cdot \PP_{r\sim\cq_n}\{i\not\in S(\cd^r)\}\right)\\
    &=p_{\cq_n},
\end{align*}
since $\PP_{r\sim\cq_n}\{i\in S(\cd^r)\}=\PP_{r\sim\cq_n}\{i\in r\}=p_{\cq_n}$.
Since this holds for every $i$, we have
$\frac{1}{n}\sum_{i\in[n]}\|\hat{w}-\hat{w}^{\backslash i}\|^2 = p_{\cq_n}^2$,
which completes the proof.
\end{proof}

\subsubsection{Remark on dimensionality}
 To relate the construction in this proof back to \Cref{asm:banach_finite_approx} and \Cref{thm:banach}, note that in the construction of $\ca$, at each sample size $n$ we are essentially working in a $K$-dimensional space for $K=2^n$---that is, even though $\cw=\Delta_\infty$ is the space of all infinite sequences, the algorithm $\ca$ only ever returns sequences $\hat{w}$ with nonzero values in the first $K=2^n$ entries (and in particular, $\hat{w}$ will always lie in the convex hull of $\mathbf{e}_1,\dots,\mathbf{e}_K$). 
 
 Compare this to the result of \Cref{thm:banach}, which essentially suggests that mean-square stability vanishes at the rate $\beta^2\propto \frac{\log K}{n}$ if $\cw$ can be approximated as a convex hull of $\mathcal{O}(K)$ many points. Since we have taken $\frac{\log K}{n}$ to be nonvanishing in this example (due to $K=2^n$ being the ``effective dimension'', as described above), stability is no longer ensured as $n\to\infty$.

\subsection{Additional results}

First, we state a result for tail stability in the Banach space setting.
\begin{corollary}\label{cor:banach_finite_B}
    Under the  notation and assumptions of \Cref{thm:banach}, derandomized bagging has tail stability $(\eps,\delta)$ with respect to $\|\cdot\|$ for any $\eps,\delta\geq 0$ satisfying
    \[\delta\eps^2\geq 1.6 \left[R\sqrt{\frac{2(1+\eta_n)\log(2K)}{n}\left(\frac{p_{\cq_n}}{1-p_{\cq_n}}+\frac{4\eta_n}{(1-p_{\cq_n})^2}\right)} + R\sqrt{\frac{1}{n} \cdot \frac{p_{\cq_n}}{1-p_{\cq_n}}} + 2\rho\cdot p_{\cq_n}\right]^2.\]
\end{corollary}
\begin{proof}
    This result follows immediately from \Cref{thm:banach} together with \Cref{prop:stability_defs}.
\end{proof}

Finally, we state a finite-$B$ result for the Banach space setting.

\begin{corollary}
    Under the  notation and assumptions of \Cref{thm:banach}, let $B\geq 1$ be fixed. Then bagging, run with base algorithm~$\ca$ and resampling distribution~$\cq_n$ and with $B$ bags, has mean-square stability $\beta^2$ with respect to~$\|\cdot\|$, with
\begin{multline*}
\beta^2
:= \inf_{\eps>0}\Bigg\{(1+\eps)\cdot 1.6 \bigg[R\sqrt{\frac{2(1+\eta_n)\log(2K)}{n}\left(\frac{p_{\cq_n}}{1-p_{\cq_n}}+\frac{4\eta_n}{(1-p_{\cq_n})^2}\right)} + R\sqrt{\frac{1}{n} \cdot \frac{p_{\cq_n}}{1-p_{\cq_n}}} \\{}+ 2\rho\cdot p_{\cq_n}\bigg]^2
{}+ (1+\eps^{-1})\cdot \left(\frac{8KR^2}{B} + 8\rho^2\right)\Bigg\}
\end{multline*}
\end{corollary}
\begin{proof}
    Let $\hat{w}$ and $\hat{w}^{\setminus i}$ be the output of the bagged algorithm (as in \Cref{def:bag}), and let $\tilde{w}$ and $\tilde{w}^{\setminus i}$ be the corresponding derandomized bagged estimators (as in \Cref{def:derand-bag}). We have
\begin{multline*}\frac{1}{n}\sum_{i\in[n]}\EE\left[\|\hat{w} - \hat{w}^{\setminus i}\|^2\right]
= \frac{1}{n}\sum_{i\in[n]}\EE\left[\|(\hat{w} - \tilde{w})  - (\hat{w}^{\setminus i} - \tilde{w}^{\setminus i}) + (\tilde{w} - \tilde{w}^{\setminus i})\|^2\right]\\
\leq \frac{1+\eps}{n}\sum_{i\in[n]}\|\tilde{w} - \tilde{w}^{\setminus i}\|^2 + \frac{2(1+\eps^{-1})}{n}\sum_{i\in[n]}\EE\left[\|\hat{w} - \tilde{w}\|^2\right] + \frac{2(1+\eps^{-1})}{n}\sum_{i\in[n]}\EE\left[\|\hat{w}^{\setminus i} - \tilde{w}^{\setminus i} \|^2 \right],\end{multline*}
since $\|a+b+c\|^2\leq (1+\eps)\|a\|^2+(1+\eps^{-1})\|b+c\|^2 \leq (1+\eps)\|a\|^2+2(1+\eps^{-1})(\|b\|^2+\|c\|^2)$, for any $a,b,c\in\cb$ (because $\|\cdot\|$ satisfies the triangle inequality).
The first term in the upper bound is bounded by \Cref{thm:banach}. Now we bound the remaining terms.

First, let $\hat{w}'$ be an i.i.d.\ copy of $\hat{w}$. Under the same notation as in the proof of \Cref{thm:banach}, let $\|\cdot\|_*$ denote the dual  norm. Then
\begin{multline*}
    \EE[\|\hat{w}-\tilde{w}\|^2]
    =\EE\left[\sup_{f:\|f\|_*\leq 1}f(\hat{w}-\tilde{w})^2\right]
    =\EE\left[\sup_{f:\|f\|_*\leq 1}f\left(\hat{w}-\EE[\hat{w}'\mid\hat{w}]\right)^2\right]\\
    =\EE\left[\sup_{f:\|f\|_*\leq 1}\EE\left[f\left(\hat{w}-\hat{w}'\right)\mid\hat{w}\right]^2\right]
    \leq \EE\left[\EE\left[\sup_{f:\|f\|_*\leq 1}f\left(\hat{w}-\hat{w}'\right)^2\mid\hat{w}\right]\right]
    =\EE[\|\hat{w}-\hat{w}'\|^2],
\end{multline*}
where the second step holds since $\hat{w}\indep \hat{w}'$ and $\EE[\hat{w}']=\EE[\hat{w}]=\tilde{w}$.
Now let 
$\hat{w}^{(b)} = \ca(\cd^{r_b};\xi_b)$,
for each $b\in[B]$, so that we have $\hat{w} = \frac{1}{B}\sum_{b=1}^B \hat{w}^{(b)}$, and similarly write
$\hat{w}'{}^{(b)}= \ca(\cd^{r_b'};\xi_b')$,
where the $r_b'$ and $\xi'_b$ samples are independently drawn, so that $\hat{w}'=\frac{1}{B}\sum_{b=1}^B \hat{w}'{}^{(b)}$.
Then $\hat{w}-\hat{w}'$ is equal in distribution to
$\frac{1}{B}\sum_{b=1}^B \zeta_b \left(\hat{w}^{(b)}-\hat{w}'{}^{(b)}\right)$
where the $\zeta_b$'s are i.i.d.\ random signs. 

Now we will bound $\EE[\|\frac{1}{B}\sum_{b=1}^B \zeta_b (w_b - w'_b)\|^2]$ for any \emph{fixed} $w_1,\dots,w_B,w'_1,\dots,w'_B\in\cw$.
For each $b$ we can write
$w_b-w'_b = \sum_{j=1}^K t^b_j v_j + \Delta_b$
with $\|t^b\|_1\leq R$ and $\|\Delta_b\|\leq\rho$, by \Cref{asm:banach_finite_approx}. Then
\begin{multline*}
    \left\|\frac{1}{B}\sum_{b=1}^B \zeta_b (w_b - w'_b)\right\|
    =\left\|\frac{1}{B}\sum_{b=1}^B \zeta_b \left(\sum_{j=1}^K t^b_j v_j + \Delta_b\right)\right\|\\
    \leq \sum_{j=1}^K \left\|\frac{1}{B}\sum_{b=1}^B\zeta_b  t^b_j v_j \right\| + \left\|\frac{1}{B}\sum_{b=1}^B \zeta_b \Delta_b\right\|
    \leq \sum_{j=1}^K \left|\frac{1}{B}\sum_{b=1}^B\zeta_b  t^b_j  \right| + \rho,
\end{multline*}
since $\|v_j\|=1$ for all $j$, and $\|\Delta_b\|\leq \rho$ for all $b$. Therefore,
\begin{multline*}
\EE\left[\left\|\frac{1}{B}\sum_{b=1}^B \zeta_b (w_b - w'_b)\right\|^2\right]
\leq \EE\left[\left(\sum_{j=1}^K \left|\frac{1}{B}\sum_{b=1}^B\zeta_b  t^b_j  \right| + \rho\right)^2\right]
\leq 2K\sum_{j=1}^K \EE\left[\left|\frac{1}{B}\sum_{b=1}^B\zeta_b  t^b_j  \right| ^2\right] + 2\rho^2\\
\leq \frac{2K}{B^2} \sum_{j=1}^K \sum_{b=1}^B (t_j^b)^2+ 2\rho^2
= \frac{2K}{B^2} \sum_{b=1}^B \|t^b\|^2_2+ 2\rho^2
\leq \frac{2K}{B^2} \sum_{b=1}^B \|t^b\|^2_1+ 2\rho^2
\leq \frac{2KR^2}{B} + 2\rho^2.
\end{multline*}

Combining all our work we then have
\begin{multline*}\EE[\|\hat{w}-\tilde{w}\|^2]
\leq \EE[\|\hat{w} - \hat{w}'\|^2]
= \EE\left[\left\|\frac{1}{B}\sum_{b=1}^B \zeta_b (\hat{w}^b - \hat{w}'{}^b)\right\|^2\right]\\
=\EE\left[\EE\left[\left\|\frac{1}{B}\sum_{b=1}^B \zeta_b (\hat{w}^b - \hat{w}'{}^b)\right\|^2 \ \middle| \ \hat{w}^1,\dots,\hat{w}^B,\hat{w}'{}^1,\dots,\hat{w}'{}^B\right]\right]
\leq  \frac{2KR^2}{B} + 2\rho^2.
\end{multline*}
The same bound holds for $\EE[\|\hat{w}^{\setminus i} - \tilde{w}^{\setminus i}\|^2]$. Combining all our calculations establishes the result. 
\end{proof}

\bibliographystyle{dcu}
\bibliography{citations}

@article{vanivcek1969approximate,
  title={Approximate spectral analysis by least-squares fit: Successive spectral analysis},
  author={Van{\'\i}{\v{c}}ek, Petr},
  journal={Astrophysics and Space Science},
  volume={4},
  pages={387--391},
  year={1969},
  publisher={Springer}
}

@article{breiman1996heuristics,
  title={Heuristics of instability and stabilization in model selection},
  author={Breiman, Leo},
  journal={The Annals of Statistics},
  volume={24},
  number={6},
  pages={2350--2383},
  year={1996},
  publisher={Institute of Mathematical Statistics}
}

@article{xu2011sparse,
  title={Sparse algorithms are not stable: A no-free-lunch theorem},
  author={Xu, Huan and Caramanis, Constantine and Mannor, Shie},
  journal={IEEE transactions on pattern analysis and machine intelligence},
  volume={34},
  number={1},
  pages={187--193},
  year={2011},
  publisher={IEEE}
}

@article{breiman1996bagging,
  title={Bagging predictors},
  author={Breiman, Leo},
  journal={Machine learning},
  volume={24},
  number={2},
  pages={123--140},
  year={1996},
  publisher={Springer}
}

@article{devroye1979distribution2,
  title={Distribution-free performance bounds for potential function rules},
  author={Devroye, Luc and Wagner, Terry},
  journal={IEEE Transactions on Information Theory},
  volume={25},
  number={5},
  pages={601--604},
  year={1979},
  publisher={IEEE}
}

@article{devroye1979distribution,
  title={Distribution-free inequalities for the deleted and holdout error estimates},
  author={Devroye, Luc and Wagner, Terry},
  journal={IEEE Transactions on Information Theory},
  volume={25},
  number={2},
  pages={202--207},
  year={1979},
  publisher={IEEE}
}

@article{mukherjee2006learning,
  title={Learning theory: stability is sufficient for generalization and necessary and sufficient for consistency of empirical risk minimization},
  author={Mukherjee, Sayan and Niyogi, Partha and Poggio, Tomaso and Rifkin, Ryan},
  journal={Advances in Computational Mathematics},
  volume={25},
  number={1},
  pages={161--193},
  year={2006},
  publisher={Springer}
}

@article{bousquet2002stability,
  title={Stability and generalization},
  author={Bousquet, Olivier and Elisseeff, Andr{\'e}},
  journal={The Journal of Machine Learning Research},
  volume={2},
  pages={499--526},
  year={2002},
  publisher={JMLR. org}
}

@article{shalev2010learnability,
  title={Learnability, stability and uniform convergence},
  author={Shalev-Shwartz, Shai and Shamir, Ohad and Srebro, Nathan and Sridharan, Karthik},
  journal={The Journal of Machine Learning Research},
  volume={11},
  pages={2635--2670},
  year={2010},
  publisher={JMLR. org}
}

@article{yu2013stability,
  title={Stability},
  author={Yu, Bin},
  journal={Bernoulli},
  volume={19},
  number={4},
  pages={1484--1500},
  year={2013},
  publisher={Bernoulli Society for Mathematical Statistics and Probability}
}

@article {kim2021black,
    AUTHOR = {Kim, Byol and Barber, Rina Foygel},
     TITLE = {Black-box tests for algorithmic stability},
   JOURNAL = {Inf. Inference},
  FJOURNAL = {Information and Inference. A Journal of the IMA},
    VOLUME = {12},
      YEAR = {2023},
    NUMBER = {4},
     PAGES = {Paper No. iaad039, 30},
      ISSN = {2049-8764,2049-8772},
   MRCLASS = {62G10 (62G20 62R07 68T05)},
  MRNUMBER = {4654900},
       DOI = {10.1093/imaiai/iaad039},
       URL = {https://doi.org/10.1093/imaiai/iaad039},
}

@article{soloff2024bagging,
  author  = {Jake A. Soloff and Rina Foygel Barber and Rebecca Willett},
  title   = {Bagging Provides Assumption-free Stability},
  journal = {Journal of Machine Learning Research},
  year    = {2024},
  volume  = {25},
  number  = {131},
  pages   = {1--35},
  url     = {http://jmlr.org/papers/v25/23-0536.html}
}

@article{hayes2005large,
  title={A large-deviation inequality for vector-valued martingales},
  author={Hayes, Thomas P},
  journal={Combinatorics, Probability and Computing},
  year={2005}
}

@article{huggins2023reproducible,
  title={Reproducible model selection using bagged posteriors},
  author={Huggins, Jonathan H and Miller, Jeffrey W},
  journal={Bayesian Analysis},
  volume={18},
  number={1},
  pages={79--104},
  year={2023},
  publisher={International Society for Bayesian Analysis}
}

@article{yu2020veridical,
  title={Veridical data science},
  author={Yu, Bin and Kumbier, Karl},
  fjournal={Proceedings of the National Academy of Sciences},
  journal={Proc. Natl. Acad. Sci},
  volume={117},
  number={8},
  pages={3920–3929}, 
  year={2020}
}

@article{buhlmann2014discussion,
  title={Discussion of big {B}ayes stories and {BayesBag}},
  author={B{\"u}hlmann, Peter},
  journal={Statistical science},
  volume={29},
  number={1},
  pages={91--94},
  year={2014},
  publisher={JSTOR}
}

@article{abadie2003economic,
  title={The economic costs of conflict: {A} case study of the Basque Country},
  author={Abadie, Alberto and Gardeazabal, Javier},
  journal={American economic review},
  volume={93},
  number={1},
  pages={113--132},
  year={2003},
  publisher={American Economic Association}
}

@article {abadie2010synthetic,
    AUTHOR = {Abadie, Alberto and Diamond, Alexis and Hainmueller, Jens},
     TITLE = {Synthetic control methods for comparative case studies:
              estimating the effect of {C}alifornia's tobacco control
              program},
   JOURNAL = {J. Amer. Statist. Assoc.},
  FJOURNAL = {Journal of the American Statistical Association},
    VOLUME = {105},
      YEAR = {2010},
    NUMBER = {490},
     PAGES = {493--505},
      ISSN = {0162-1459,1537-274X},
   MRCLASS = {99-01},
  MRNUMBER = {2759929},
       DOI = {10.1198/jasa.2009.ap08746},
}

@article{abadie2015comparative,
  title={Comparative politics and the synthetic control method},
  author={Abadie, Alberto and Diamond, Alexis and Hainmueller, Jens},
  journal={American Journal of Political Science},
  volume={59},
  number={2},
  pages={495--510},
  year={2015},
  publisher={Wiley Online Library}
}

@article{abadie2021using,
  title={Using synthetic controls: {F}easibility, data requirements, and methodological aspects},
  author={Abadie, Alberto},
  journal={Journal of Economic Literature},
  volume={59},
  number={2},
  pages={391--425},
  year={2021},
  publisher={American Economic Association 2014 Broadway, Suite 305, Nashville, TN 37203-2425}
}

@article{murdoch2019definitions,
  title={Definitions, methods, and applications in interpretable machine learning},
  author={Murdoch, W James and Singh, Chandan and Kumbier, Karl and Abbasi-Asl, Reza and Yu, Bin},
  journal={Proceedings of the National Academy of Sciences},
  volume={116},
  number={44},
  pages={22071--22080},
  year={2019},
  publisher={National Acad Sciences}
}

@article{barber2024hoeffding,
  title={Hoeffding and {B}ernstein inequalities for weighted sums of exchangeable random variables},
  author={Barber, Rina Foygel},
  journal={arXiv preprint arXiv:2404.06457},
  year={2024}
}

@book{yu2024,
  title     = "Veridical Data Science: The Practice of Responsible Data Analysis and Decision
Making",
  author    = "Yu, Bin and Barter, Rebecca",
  year      = 2024,
  publisher = "MIT Press"
}

@book {folland1999analysis,
    AUTHOR = {Folland, Gerald B.},
     TITLE = {Real analysis},
    SERIES = {Pure and Applied Mathematics (New York)},
   EDITION = {Second},
      NOTE = {Modern techniques and their applications,
              A Wiley-Interscience Publication},
 PUBLISHER = {John Wiley \& Sons, Inc., New York},
      YEAR = {1999},
     PAGES = {xvi+386},
      ISBN = {0-471-31716-0},
   MRCLASS = {00A05 (26-01 28-01 46-01)},
  MRNUMBER = {1681462},
}

\end{document}